\documentclass[11pt,reqno]{amsart}

\usepackage{amsmath,amsthm,amssymb,thmtools,tikz-cd,comment,hyperref,cleveref,nicematrix,enumitem,breqn,float,caption,kotex,graphicx,todonotes}

\allowdisplaybreaks
\theoremstyle{plain}
\newtheorem{theorem}{Theorem}[section]
\newtheorem{lemma}[theorem]{Lemma}
\newtheorem{proposition}[theorem]{Proposition}
\newtheorem{corollary}[theorem]{Corollary}
\newtheorem{thmalphabetintro}{Theorem}

\newtheorem{thmalphabetmaintext}{Theorem}

\newtheorem{coralphabetintro}[thmalphabetintro]{Corollary}

\newtheorem{coralphabetmaintext}[thmalphabetmaintext]{Corollary}

\theoremstyle{definition}
\newtheorem{definition}[theorem]{Definition}
\newtheorem{example}[theorem]{Example}

\newtheorem{remark}[theorem]{Remark}

\newtheorem{problem}[theorem]{Problem}
\newtheorem{open problem}[theorem]{Open Problem}
\newtheorem{open question}[theorem]{Open question}
\newtheorem{remark and notation}[theorem]{Remark and Notation}

\def \p {\mathbb{P}}
\def \g {\mathbb{G}}
\def \z {\mathbb{Z}}

\def \c {\mathbb{C}}

\def \N {\textnormal{\bf{N}}}
\def \K {\textnormal{\bf{K}}}

\def \sing {\textnormal{sing}}

\def \leng {\operatorname{length}}
\def \supp {\operatorname{Supp}}

\def \tan {\operatorname{Tan}}
\def \bl {\operatorname{Bl}}
\def \rank {\operatorname{rank}}
\def \proj {\operatorname{Proj}}
\def \codim {\operatorname{codim}}
\def \depth {\operatorname{depth}}
\def \tor {\operatorname{Tor}}

\def \Span {\operatorname{Span}}
\def \maxdeg {\operatorname{maxdeg}}

\def \depth {\operatorname{depth}}
\def \seg {\operatorname{Seg}}
\def \Sym {\operatorname{Sym}}

\def \wt {\widetilde}

\title[Linear syzygies and multisecants]{Linear syzygies and linear subspaces whose lines are multisecant}

\author{Jong In Han}
\address{Jong In Han, School of Mathematics, Korea Institute for Advanced Study (KIAS), 85 Hoegi-ro, Dongdaemun-gu, Seoul, Republic of Korea}
\email{jihan09@kias.re.kr}

\author{Sijong Kwak}
\address{Sijong Kwak, Department of Mathematical Sciences, Korea Advanced Institute of Science and Technology (KAIST), 291, Daehak-ro, Yuseong-gu, Daejeon, Republic of Korea} 
\email{sjkwak@kaist.ac.kr}

\thanks{}

\begin{document}
	\begin{abstract}
		The locus $S_d(X)$ of $d$-secant lines to a projective variety $X$ plays an important role in the study of projective varieties via projections \cite{MR894589,MR1620706,MR2601633}. For a projective scheme $X\subseteq\mathbb{P}^r$ satisfying $\textbf{N}_{d,2}$, we show that $S_d(X)\cup X$ is cut out set-theoretically by the $r$-minors of a matrix $M$ constructed from $d$-forms scheme-theoretically defining $X$ and their linear syzygies. In particular, $S_d(X)=\mathbb{P}^r$ if and only if every $r$-minor of $M$ vanishes. Using this, we find a singular threefold $X\subseteq\mathbb{P}^5$ with $S_4(X)\ne \mathbb{P}^5$ and $(I_X)_3=0$, and a normal fourfold $X\subseteq\mathbb{P}^7$ with $S_3(X)\ne \mathbb{P}^7$ and $(I_X)_2=0$. The nonexistence of such varieties in the smooth case is an open question raised by the second author and by Gruson--Peskine. Next, we consider the locus $S_{k,d}(X)$ of $k$-planes $L$ such that $L\cap X$ contains a hypersurface of degree $\ge d$ in $L$ so that $S_{1,d}(X)=S_d(X)$. The locus $S_{k,d}(X)\cup X$ is set-theoretically defined by the $(r+1-k)$-minors of $M$. As a result, we obtain determinantal equations vanishing on the $(q+1)$-secant variety $\sigma_{q+1}X$ from the equations of $\sigma_qX$ and their linear syzygies when $\sigma_qX$ satisfies $\textbf{N}_{q+1,2}$. Finally, we extend the $d$-secant lemma for lines to the locus $S_{k,d}(X)$ of $k$-planes. This yields a lower bound on the number of linear syzygies.
		%
		%
		% abstract without \cite
		%The locus $S_d(X)$ of $d$-secant lines to a projective variety $X$ plays an important role in the study of projective varieties via projections (Lazarsfeld 1987, Kwak 1998, Beheshti-Eisenbud 2010). For a projective scheme $X\subseteq\mathbb{P}^r$ satisfying $\textbf{N}_{d,2}$, we show that $S_d(X)\cup X$ is cut out set-theoretically by the $r$-minors of a matrix $M$ constructed from $d$-forms scheme-theoretically defining $X$ and their linear syzygies. In particular, $S_d(X)=\mathbb{P}^r$ if and only if every $r$-minor of $M$ vanishes. Using this, we find a singular threefold $X\subseteq\mathbb{P}^5$ with $S_4(X)\ne \mathbb{P}^5$ and $(I_X)_3=0$, and a normal fourfold $X\subseteq\mathbb{P}^7$ with $S_3(X)\ne \mathbb{P}^7$ and $(I_X)_2=0$. The nonexistence of such varieties in the smooth case is an open question raised by the second author and by Gruson-Peskine. Next, we consider the locus $S_{k,d}(X)$ of $k$-planes $L$ such that $L\cap X$ contains a hypersurface of degree $\ge d$ in $L$ so that $S_{1,d}(X)=S_d(X)$. The locus $S_{k,d}(X)\cup X$ is set-theoretically defined by the $(r+1-k)$-minors of $M$. As a result, we obtain determinantal equations vanishing on the $(q+1)$-secant variety $\sigma_{q+1}X$ from the equations of $\sigma_qX$ and their linear syzygies when $\sigma_qX$ satisfies $\textbf{N}_{q+1,2}$. Finally, we extend the $d$-secant lemma for lines to the locus $S_{k,d}(X)$ of $k$-planes. This yields a lower bound on the number of linear syzygies.
	\end{abstract}
	\maketitle
	
	\section{Introduction}\label{sec_introduction}
	
	Let $X\subseteq\p^r=\proj \c[x_0,\dots,x_r]$ be a nondegenerate smooth projective variety over $\c$ of dimension $n$ with the homogeneous ideal $I_X$ and the homogeneous coordinate ring $S_X$. Then a line $l\subseteq\p^r$ is called a $d$-secant line of $X$ if \[\leng(l\cap X)\geq d.\] We denote the locus of $d$-secant lines of $X$ as $S_d(X)$.
	We also regard lines lying in $X$ as $d$-secant lines.
	Classically, the trisecant lemma states that for any nondegenerate smooth projective curve $C\subseteq\p^r$ ($r\geq 3$), it holds that
	\[
		\dim S_3(C)\le 2.
	\]
	In 1991, Ran \cite{MR1123373} generalized the trisecant lemma to the (dimension$+2$)-secant lemma: for any smooth $n$-fold $X$, it holds that
	\[
		\dim S_{n+2}(X)\le n+1.
	\]
	In 2010, Beheshti and Eisenbud \cite{MR2601633} generalized it further to the $d$-secant lemma: for any smooth $n$-fold $X$, it holds that
	\[
		\dim S_{d}(X)\le n+1+\dfrac{n}{d-1}.
	\]
	In 2013, Gruson and Peskine \cite[p. 557]{MR3024093} gave another proof of this using the aligned Hilbert scheme theorem, and in 2017, Bastianelli, De Poi, Ein, Lazarsfeld, and Ullery \cite{MR3705293} gave another proof of this.
	The statement proved in \cite{MR3705293} allows $X$ to have singularities (\Cref{thm_k-secant_lemma_singular_case}).
	
	However, despite these bounds, even the dimension of such a locus is hard to compute. The following problem, which naturally arises in the study of Castelnuovo--Mumford regularity via general projections, is of particular interest (see e.g. \cite{MR894589,MR1064868,MR1620706}).
	
	\begin{problem}\label{problem_fill_up_general}
		Classify nondegenerate smooth $n$-folds $X\subseteq\p^r$ such that $(I_X)_{d-1}=0$ and $S_d(X)\neq\p^r$.
	\end{problem}
	
	The assumption $(I_X)_{d-1}=0$ is necessary since if $(I_X)_{d-1}\ne 0$, then it trivially holds that $S_d(X)\neq\p^r$ since $S_d(X)$ lies in the hypersurface defined by $f\in (I_X)_{d-1}\setminus \{0\}$.
	This problem is studied not only for its own sake but also for its connection to the Eisenbud--Goto conjecture \cite{MR741934}.
	For example, the above results give $\dim S_4(X)\le 5$, and if one can prove $\dim S_4(X)\le 4$ for a smooth threefold $X$, then the Eisenbud--Goto conjecture holds for such $X$ (cf. \cite{MR894589,MR1064868}).
	Although the Eisenbud--Goto conjecture has singular counterexamples \cite{MR3758150}, the smooth case remains open.

	To study the locus of $d$-secant lines, one may ask how to compute such a locus for a given $X$.
	In general, this is a difficult question to answer.
	Studying this question, we are naturally led to the following locus.
	{\small
	\[
	S_{k,d}(X) := \left\{
	\begin{aligned}
		&\text{Union of $k$-planes $L$ such that $L\cap X$ scheme-}\\
		&\text{theoretically contains a hypersurface of degree $\geq d$ in $L$}
	\end{aligned}
	\right\}\subseteq\p^r
	\]%
	}%
	This locus is closed (\Cref{prop:Skd_closed}).
	Note that when $k=1$, such $1$-planes are exactly the $d$-secant lines, hence
	\[
		S_{1,d}(X)=S_d(X).
	\]

	With this definition, we give set-theoretic equations of $S_{k,d}(X)\cup X$ when $X$ satisfies the property $\N_{d,2}$.
	Here, we say a projective scheme $X\subseteq\p^r$ satisfies the property $\N_{d,p}$ if for all $i\le p$ and $j\ge d$, the $(i,j)$-th Betti number vanishes, i.e., $\beta_{i,j}(X):=\dim_\c\tor_i(S_X,\c)_{i+j}=0$ where $S_X$ denotes the homogeneous coordinate ring of $X$ \cite{MR2188445}. Similarly, we say $d$-forms $F_0,\dots,F_s$ satisfy $\N_{d,p}$ if the ideal $I$ generated by those satisfies $\dim_\c\tor_i(S/I,\c)_{i+j}=0$ for all $i\le p$ and $j\ge d$.
	When $X$ satisfies $\N_{d,2}$, it is scheme-theoretically defined by $d$-forms satisfying $\N_{d,2}$. Indeed, we may replace any $(d-1)$-form $F$ in the minimal generating set of $I_X$ by $x_0F,\dots,x_rF$, and apply a similar process to the forms of lower degrees in the generating set. The resulting generators, consisting of $d$-forms only, define $X$ scheme-theoretically and still satisfy $\N_{d,2}$.
	In the following, $I_t(M)$ denotes the ideal of $t$-minors of $M$ and $\K_d$ is a property defined by Vermeire (see \Cref{def_Kd}).
	
	\begin{thmalphabetintro}[Set-theoretic equations of $S_{k,d}(X)\cup X$]\label{intro_set-theoretic_eq}
		Let $X\subseteq\p^r=\textup{Proj}\ \c[x_0,\dots,x_r]$ be a (possibly singular) projective scheme that is scheme-theoretically defined by $d$-forms $F_0,\dots,F_s$ satisfying $\N_{d,2}$ (or more generally $\K_d$). Let \[\{(L_{i0},\dots,L_{is})\}_{1\leq i\leq m}\] be a generating set of the linear syzygies of the $F_i$'s, and write
		\[
		L_{ij} =a^0_{ij}x_0+\cdots+a^r_{ij}x_r
		\]
		where $a^k_{ij}\in\c$.
		Let $M$ be the $m\times (r+1)$ matrix of $d$-forms as follows.
		\[M=\begin{bmatrix}
			\sum\limits^s_{j=0} a^0_{1j}F_j &\sum\limits^s_{j=0} a^1_{1j}F_j& \cdots & \sum\limits^s_{j=0} a^r_{1j}F_j\\
			\sum\limits^s_{j=0} a^0_{2j}F_j &\sum\limits^s_{j=0} a^1_{2j}F_j& \cdots & \sum\limits^s_{j=0} a^r_{2j}F_j\\
			\vdots & \vdots & \ddots & \vdots \\
			\sum\limits^s_{j=0} a^0_{mj}F_j &\sum\limits^s_{j=0} a^1_{mj}F_j& \cdots & \sum\limits^s_{j=0} a^r_{mj}F_j\\
		\end{bmatrix}\]
		Then $I_{r+1-k}(M)$ defines $S_{k,d}(X)\cup X$ set-theoretically for every $k\geq 1$.
		In particular, $I_r(M)$ defines $S_{1,d}(X)\cup X$ set-theoretically.
		Hence $S_{1,d}(X)=\p^r$ if and only if every $r$-minor of $M$ is zero.
	\end{thmalphabetintro}

	The proof of \Cref{intro_set-theoretic_eq} is based on Vermeire's description of the fibers of the rational map $\p^r\dashrightarrow\p^s$ defined by $F_0,\dots,F_s$ (see \Cref{prop_fiber_of_lifted_kd_map}).

	Now we return to \Cref{problem_fill_up_general} and recall the history of the codimension two case.
	Due to Zak \cite[Corollary II.2.11]{MR1234494}, when $n> \dfrac{2}{3}(r-2)$, we have $S_{2}(X)=\p^r$.
	Consequently, we have $S_{2}(X)=\p^{n+2}$ in the case of codimension two.
	Furthermore, the (dimension+2)-secant lemma shows $S_{d}(X)\neq \p^{n+2}$ for all $d\geq n+2$.
	Hence it is enough to consider the case $3\leq d\leq n+1$.
	The case $n\geq 4$ was solved by Ran \cite{MR714095}, the case $n=2$ was solved by Aure \cite{MR962514}\footnote{Severi had also studied this case \cite{zbMATH02663596}.}, and the case $n=3,d=3$ was solved by the second author \cite{MR1857133}.
	Hence the case $n=3,d=4$ is the only unsolved case for codimension two.
	This makes the case interesting, and it is covered in various papers, see \cite{MR1857133,MR1984214}, and also \cite[p. 575]{MR3024093} for a related discussion.
	
	For the case of codimension greater than two, almost nothing is known. As Gruson and Peskine pointed out, no smooth variety $X$ of dimension $2n$ in $\p^{3n+1}$ ($n\geq 2$) is known such that $(I_X)_2=0$ and $S_{3}(X)\neq\p^{3n+1}$ \cite[p. 575]{MR3024093}. Therefore, the most basic question that has to be resolved currently is the following.
	
	\begin{open question}[{\cite[Open questions 4.7]{MR1857133},\cite[pp. 574-575]{MR3024093}}]\label{question:conjecture}
		Is there any nondegenerate smooth $n$-fold $X\subseteq\p^r$ satisfying 
		\[S_{d}(X)\neq\p^r\text{ and }(I_X)_{d-1}=0\]
		when
		\begin{enumerate}
			\item $n=3$, $r=5$, $d=4$, or \label{question:threefold}
			\item $n=4$, $r=7$, $d=3$? \label{question:fourfold}
		\end{enumerate}
	\end{open question}

	Currently, no such varieties are known in either case. To the best of the authors' knowledge, even the existence of such singular varieties was not known. Meanwhile, one can determine whether $S_{d}(X)\neq\p^r$ very fast using \Cref{intro_set-theoretic_eq}: one can evaluate the rank of the matrix at a point that is a certificate of $S_{d}(X)\neq\p^r$ (\Cref{rem_determine_probabilistically}).
	Using this, we get the following corollary.
	
	\begin{coralphabetintro}\label{cor:existence}
		There exist the following nondegenerate projective varieties:
		\begin{enumerate}
			\item a threefold $X\subseteq\p^5$ regular in codimension $1$ with $S_{4}(X)\neq \p^5$ and $(I_X)_3=0$,
			\item a normal fourfold $X\subseteq\p^7$ with $S_{3}(X)\neq\p^7$ and $(I_X)_2=0$.
		\end{enumerate}
	\end{coralphabetintro}

	In particular, if the smoothness condition in \Cref{question:conjecture} is weakened to Serre's condition $R_1$ (resp. the normality condition) in Case \eqref{question:threefold} (resp. Case \eqref{question:fourfold}), then the answer is positive.
	
	We give another application of \Cref{intro_set-theoretic_eq}.
	Denote the $q$-secant variety of $X$ as
	\[
	\sigma_qX:=\overline{\bigcup_{p_i\in X}\langle p_1,\dots,p_q\rangle}.
	\]
	In \cite[Lemma 4.1]{MR1985451}, Vermeire showed $\sigma_{q+1}X\subseteq S_{q+1}(\sigma_q X)$.
	With the locus $S_{k,d}$ defined above, we make a finer filtration
	\[
	\sigma_{q+1}X\subseteq S_{q,q+1}(\sigma_{q}X)\subseteq S_{q-1,q+1}(\sigma_{q}X)\subseteq\cdots\subseteq S_{1,q+1}(\sigma_{q}X)=S_{q+1}(\sigma_{q}X).
	\]
	Consequently, we get the following corollary of \Cref{intro_set-theoretic_eq}.
	
	\begin{coralphabetintro}\label{intro_secant_eq}
		Let $X\subseteq\p^r$ be a nondegenerate projective variety such that $\sigma_qX\subsetneq\p^r$ satisfies $\N_{q+1,2}$ where $q\ge 1$.
		Define $M$ as in \Cref{intro_set-theoretic_eq} for the minimal generators of $I_{\sigma_qX}$. Then it holds that
		\[
			I_{r-q+1}(M)\subseteq I_{\sigma_{q+1}X}.
		\]
	\end{coralphabetintro}

	The condition that $\sigma_qX\subsetneq\p^r$ satisfies $\N_{q+1,2}$ holds in natural cases.
	Indeed, for a given integer $q\ge 1$ and for a smooth $n$-fold $X$ that is embedded by the complete linear system of a sufficiently positive line bundle, its secant variety $\sigma_qX$ satisfies $\N_{q+1,2}$ when $n\le 2$ or $q\le 3$ \cite[Theorem B]{CLPS25}.
	
	On the other hand, we show that there is a dimension bound on $S_{k,d}(X)$ when $X$ is smooth, without the condition $\N_{d,2}$ required before.
	To achieve this, we naturally prove a slightly more general bound, that is, a bound for $\dim S^*_{k,d}(X)$ where $S^*_{k,d}(X)\subseteq S_{k,d}(X)$ is the union of $k$-planes $L$ such that
	\begin{itemize}
		\item $L\cap X$ scheme-theoretically contains a hypersurface of degree $\geq d$ in $L$, as before, and
		\item $\dim (L\cap X_\sing)\le k-2$.
	\end{itemize}
	Note that when $\dim X_\sing\le k-2$, in particular if $X$ is smooth, it holds that
	\[
		S^*_{k,d}(X)=S_{k,d}(X).
	\]
	\begin{thmalphabetintro}\label{intro_dimension_bound}
		Let $X\subseteq\p^r$ be a projective variety of dimension $n$ and let $1\le k\le n+1$ and $d\ge 2$. Then \[\dim S^*_{k,d}(X)\leq \max(n+1+\dfrac{n+1-k}{d-1},n+k).\]
	\end{thmalphabetintro}
	When $k=1$, this recovers the $d$-secant lemma in \cite{MR2601633} (smooth case) and \cite{MR3705293} (singular case).
	
	As a corollary, we get the following lower bound on the number of linear syzygies.
	\begin{coralphabetintro}\label{cor_lower_bound_beta_2d-1_intro}
		Let $X\subseteq\p^r$ be a projective variety of dimension $n$ and let $1\le k\le n+1$ and $d\ge 2$.
		Suppose that the homogeneous ideal $I_X$ is generated by $d$-forms satisfying $\N_{d,2}$.
		If $r>n+1+\dfrac{n+1-k}{d-1}$ and $\dim X_{\sing}\leq k-2$, then $\beta_{2,d-1}(X)\geq r+1-k$.
	\end{coralphabetintro}
	This bound is sharp, attained by certain secant varieties of Segre varieties (\Cref{exp_sharp_cases_of_lower_bound}), so it cannot be improved in general.
	\\
	\\
	\indent
	This paper is organized as follows. In \Cref{sec_preliminaries}, we give preliminaries. In \Cref{sec:Skd_eqns}, we give set-theoretic equations of $S_{k,d}(X)$ and prove \Cref{intro_set-theoretic_eq} and \Cref{intro_secant_eq}. In \Cref{sec_dimension_bound}, we prove the dimension bound (\Cref{intro_dimension_bound}). In \Cref{sec:applications}, we give applications of main results. In \Cref{sec_Multisecant Lines to the Varieties of Codimension Two}, we show \Cref{cor:existence}.
	The computations in this paper are done by Macaulay2 \cite{M2} and the code can be found in \url{https://github.com/jihan099/multisecant}.
	
	\section*{Acknowledgements}
	J. I. Han was partially supported by Basic Science Research Program through the National Research Foundation of Korea (NRF) funded by the Ministry of Education (2019R1A6A1A10073887) when he was at Korea Advanced Institute of Science and Technology (KAIST). J. I. Han is currently supported by a KIAS Individual Grant (MG101401) at Korea Institute for Advanced Study.
	S. Kwak was supported by the National Research Foundation of Korea (NRF) grant funded	by the Korea government (MSIT) (No. RS-2024-00352592).
	The authors thank Junho Choe for helpful discussions. The proof of \Cref{lem_general_projection_preserves_dim_skd} is based on his argument given when we were reading \cite{MR3705293}.
	
	\section{Preliminaries}\label{sec_preliminaries}
	
	In this paper, we work over $\c$ and require a variety to be integral. We use the following notation unless we state otherwise.
	\begin{itemize}
		\item The homogeneous coordinate ring of $X$ is denoted as $S_X$.
		\item The polynomial ring $\c[x_0,\dots,x_r]$ of $\p^r$ is denoted as $S$.
		\item The $(i,j)$-th graded Betti number of $X$ is denoted by \[\beta_{i,j}(X):=\dim\tor^S_i(S_X,\c)_{i+j}.\]
		\item The image of the projection of $X\subseteq\p^r$ from the point $q\in \p^r$ to $\p^{r-1}$ is denoted as $X_q$.
		\item The image of a rational map is defined as the closure of its set-theoretic image.
		\item The depth of $X$ is defined as the depth of $S_X$.
	\end{itemize}
	
	We often identify a point $l\in\g(1,r)$ with the corresponding line $l\subseteq\p^r$.
	We set $\dim\varnothing=-1$.
	
	\begin{definition}\label{def_sd*}
		For a projective scheme $X\subseteq \p^r$ and $d\geq 2$, define
		\begin{itemize}
			\item $\Sigma_{d}(X) =\{l\in\g(1,r)\mid \leng(l\cap X)\geq d\}$,\\
			\item $\Sigma^*_{d}(X) =\{l\in\g(1,r)\mid \leng(l\cap X)\geq d,\,\supp(l\cap X)\subseteq X_{\text{sm}}\}$,\\
			\item $S_{d}(X) =\bigcup\limits_{l\in\Sigma_{d}(X)}l\subseteq\p^r$, and\\
			\item $S^*_{d}(X) =\bigcup\limits_{l\in\Sigma^*_{d}(X)}l\subseteq\p^r$.
		\end{itemize}
		The elements in $\Sigma_{d}(X)$ are called the \textit{$d$-secant lines} of $X$.
	\end{definition}
	
	Note that if $X$ is smooth, then \[\Sigma^*_{d}(X)=\Sigma_{d}(X)\text{ and }S^*_{d}(X)=S_{d}(X).\]
	
	\subsection{The $d$-secant lemma}\label{subsec_k-secant_lemma}	
	The following theorem, the $d$-secant lemma, is proved in \cite{MR2601633} for the smooth case, and in \cite{MR3705293} for the singular case.
	\begin{theorem}[\cite{MR2601633}, \cite{MR3705293}]\label{thm_k-secant_lemma_singular_case}
		Let $X\subseteq \p^r$ be a projective variety of dimension $n$. Then \[\dim S^*_{d}(X)\leq n+1+\dfrac{n}{d-1}.\]
	\end{theorem}
	
	We extend this to the locus of $k$-planes in Section \ref{sec_dimension_bound}.
	For that purpose, we use the following results.
	
	\begin{theorem}[{\cite[Theorem 2.4]{MR2441617}}]\label{lem_focal_locus_is_hypersurface_of_degree_r-k}
		Let $\Sigma\subseteq\g(k,r)$ be a congruence of $k$-planes in $\p^r$, i.e. an $(r-k)$-dimensional subvariety of $\g(k,r)$.
		Let $\Gamma\subseteq\Sigma\times\p^r$ be the incidence correspondence.
		\[
		\begin{tikzcd}
			\Gamma \ar[r,"\pi_2"]\ar[d,"\pi_1"] &\p^r\\
			\Sigma
		\end{tikzcd}
		\]
		Let $R$ be the ramification divisor of $\pi_2$, and pick any $L\in\Sigma$.
		Then the focal locus for $L$, i.e. the restriction of $R$ to $\pi_1^{-1}(L)$, is either a hypersurface of degree $r-k$ in $L$, or the entire $L$.
	\end{theorem}
	
	\begin{lemma}[{\cite[Lemma A.5]{MR3705293}}]\label{lem_order_of_ramification_divisor}
		Let $\Gamma$ and $P$ be $r$-dimensional smooth projective varieties.
		Let $\mu:\Gamma\to P$ be a surjective generically finite morphism.
		Let $X$ be a codimension $c$ subvariety of $P$, and let $E_0$ be a prime divisor of $\Gamma$ such that $\mu(E_0)\subseteq X$ and $\mu(E_0)\cap X_{\textnormal{sm}}\neq \varnothing$.
		Let $\mathcal{I}_X$ be the ideal sheaf of $X$ in $P$, and define $a:=\textup{ord}_{E_0}(\mathcal{I}_X\cdot \mathcal{O}_\Gamma)$.
		Denote the ramification divisor of $\mu$ as $R$.
		Then $\textup{ord}_{E_0}(R)\geq ca-1$.
	\end{lemma}
	
	\subsection{The fibers of the morphism defined by $d$-forms satisfying $\K_d$}\label{sec_vermeire}
	
	In this subsection, we review some results on the fibers of the morphism defined by $d$-forms satisfying $\K_d$, as studied by Vermeire \cite{MR1818982,MR1985451}.
	
	\begin{definition}
		For homogeneous polynomials $F_0,\dots,F_s$ in a polynomial ring, the \textit{trivial syzygies of $F_0,\dots,F_s$} are the syzygies of the form
		\[(0,\dots,F_j,\dots,-F_i,\dots,0)\]
		corresponding to the trivial relation $F_jF_i+(-F_i)F_j=0$.
	\end{definition}
	
	\begin{definition}\label{def_Kd}
		Let $X$ be a projective scheme. We say \textit{$(X,F_0,\dots,F_s)$ satisfies the property $\K_d$} if $X$ is scheme-theoretically defined by $d$-forms $F_0,\dots,F_s$ whose trivial syzygies are generated by linear syzygies. We say \textit{$X$ satisfies the property $\K_d$} if there exist such $d$-forms. 
	\end{definition}
	
	Let $X\subseteq\p^r$ be a projective scheme satisfying $\K_{d}$ by $d$-forms $F_0,\dots,F_s$.
	Let $V\subseteq H^0(\p^r,\mathcal{O}_{\p^r}(d))$ be the linear system spanned by $F_0,\dots,F_s$ and $\varphi:\p^r\backslash X\to \p^s$ be the corresponding map.
	
	\begin{remark}\label{rem_k-plane_maps_to_a_point}
		Suppose a $k$-plane $L$ intersects $X$ in a hypersurface of degree $d$ of $L$, and let $f\in H^0(L,\mathcal{O}_L(d))$ be the element that defines the hypersurface.
		Since $\varphi$ is defined by a subsystem of $H^0(\p^r,\mathcal{O}_{\p^r}(d))$, the morphism $\varphi|_L$ should be defined by a subsystem of $H^0(L,\mathcal{O}_{L}(d))$ with the base scheme $L\cap X$.
		Hence the linear system corresponding to $\varphi|_L$ is $$\Span(f)\subseteq H^0(L,\mathcal{O}_L(d)).$$ Thus $L$ maps to a point under $\varphi$.
	\end{remark}
	
	Let $\Gamma_{\varphi}\subseteq\p^r\times\p^s$ be the closure of the graph corresponding to the rational map $\varphi:\p^r\dashrightarrow\p^s$.
	Let $\pi_1:\Gamma_\varphi\to\p^r$ and $\pi_2:\Gamma_\varphi\to\p^s$ be the induced projections.
	Then $\pi_1:\Gamma_\varphi\to\p^r$ is in fact the blowing up of $\p^r$ along $X$. Let $E$ be the exceptional divisor of the blowing up.
	
	\[
	\begin{tikzcd}
		\bl_X\p^r=\Gamma_\varphi \ar[r,"\pi_2"]\ar[d,"\pi_1"] &\p^s \\
		\p^r \ar[ur,"\varphi",dotted,swap]
	\end{tikzcd}
	\]
	
	Then we have the following result.
	\begin{proposition}[{\cite[Proposition 2.8]{MR1818982}}]\label{prop_fiber_of_lifted_kd_map}
		Let $X\subseteq\p^r$ be a projective scheme satisfying $\K_d$ not containing a line. Pick $a\in\pi_2(\Gamma_\varphi)$. Then \[\pi_2^{-1}(a)=\p^k\times\{a\}\subseteq\p^r\times\p^s\] where
		\begin{enumerate}[label=(\alph*)]
			\item $k=0$, or
			\item $\pi_1(\p^k\times\{a\})=\p^k\subseteq\p^r$ intersects $X$ in a hypersurface of degree $d$ in $\p^k$.
		\end{enumerate}
	\end{proposition}
	
	\begin{proof}
		Let $[x_0,\dots,x_r;y_0,\dots,y_s]$ be the bihomogeneous coordinates of $\p^r\times\p^s$. Let \[\{(L_{i0},\dots,L_{is})\}_{1\leq i\leq m}\] be the linear syzygies of $F_0,\dots,F_s$ that generate the trivial syzygies of $F_i$'s. Let $T$ and $S$ be the subschemes of $\p^r\times\p^s$ defined by the bilinear forms $\{\sum\limits_{j=0}^s L_{ij}y_j\}_{1\leq i\leq m}$ and the forms $\{F_iy_j-F_jy_i\}_{i\neq j}$ of bidegree $(d,1)$, respectively. Then it is easy to see $\Gamma_\varphi\subseteq T\subseteq S$ and they are all identified off of $E$ as schemes.
		
		Pick $a\in\pi_2(\Gamma_\varphi)$. Let $T_a$ be the fiber over $a$ of $T\to\p^s$.
		We may assume $a=[1,0,\dots,0]\in\p^s$ by a coordinate change. Then $T_a$ is scheme-theoretically defined by the bihomogeneous forms $$\{\sum\limits_{j=0}^s L_{ij}y_j\}_{1\leq i\leq m}\cup\{y_1,\dots,y_s\},$$ hence defined by $\{L_{10}y_0,\dots,L_{m0}y_0,y_1,\dots,y_s\}$. Therefore \[T_a=\p^k\times[1,0,\dots,0]\subseteq\p^r\times\p^s\] where $\p^k$ is the linear space in $\p^r$ defined by $\{L_{i0}\}_{1\leq i\leq m}$. If $k=0$, then we are done. Suppose $k\geq 1$.
		Then $\p^k$ is not contained in $X$ as $X$ does not contain a line. Hence $T_a$ is not contained in $E$. With the fact that $\pi_2^{-1}(a)=T_a$ off of $E$, we get $\pi_2^{-1}(a)=T_a$. As it maps to a point $a$, the restriction of $\varphi$ to $\pi_1(\pi_2^{-1}(a))=\p^k$ corresponds to the linear system spanned by an element in $H^0(\p^k,\mathcal{O}_{\p^k}(d))$. Hence $\pi_1(\pi_2^{-1}(a))=\p^k$ intersects $X$ in a hypersurface of degree $d$ in $\p^k$.
	\end{proof}
	
	\begin{remark}[{cf. \cite[Remark 2.11]{MR1818982}}]\label{rem_fiber_of_lifted_kd_map_without_assumption}
		Note that if we drop the condition that $X$ does not contain a line, then now there are three possibilities:
		\begin{enumerate}[label=(\alph*)]
			\item $k=0$, or
			\item $\pi_1(\p^k\times\{a\})=\p^k\subseteq\p^r$ intersects $X$ in a hypersurface of degree $d$ in $\p^k$, or
			\item $\pi_2^{-1}(a)$ is contained in $E$.
		\end{enumerate}	
	\end{remark}
	
	\subsection{d-ple point loci of outer projections}
	\begin{definition}
		Let $X\subseteq\p^r$ be a projective variety and $\pi_p:X\to\p^{r-1}$ be an outer projection of $X$ from a point $p\in\p^r\backslash X$.
		We denote its image as $X_p$ and define the \textit{$d$-ple point locus of $\pi_p$} to be the set
		\[
			Z_d(p,X):=\{x\in\p^{r-1}\mid\leng\pi_p^{-1}(x)\geq d\}\subseteq\p^{r-1},
		\]
		and the \textit{$d$-secant locus of $\pi_p$} to be the set
		\[
			\wt{Z}_d(p,X):=\{x\in X\mid\leng\pi_p^{-1}(\pi_p(x))\geq d	\}\subseteq X\subseteq\p^r.
		\]
	\end{definition}
	
	When $X\subseteq\p^r$ satisfies $\N_{d,2}$, it is known that the $d$-ple point locus and $d$-secant locus have nice properties.
	We define $\tan(X):=\bigcup_{x\in X}T_xX\subseteq\p^r$.
	
	\begin{theorem}[{\cite[Theorem 3.3]{MR4633626}}]\label{thm_d-ple_linear}
		Let $X\subseteq\p^r$ be a nondegenerate variety satisfying $\N_{d,2}$ ($d\ge 2$) and pick $p\in \left(\sigma_2X\cup \tan(X)\right)\backslash X$.
		Then
		\begin{enumerate}
			\item $Z_d(p,X)$ is either empty or a linear space, and
			\item $\wt{Z}_d(p,X)$ is a hypersurface of degree $d$ in the linear span $\langle p,Z_d(p,X)\rangle$ unless $Z_d(p,X)$ is empty.
		\end{enumerate}
	\end{theorem}

	In particular, when $d=2$, the secant locus $\wt{Z}_2(p,X)$ of $\pi_p$ plays an important role when we determine projective invariants of $X_p$.
	\begin{proposition}[{\cite[Proposition 4.11]{MR2774656}}]\label{prop_Xp_using_secant_locus}
		Let $X\subseteq\p^r$ be a nondegenerate variety of dimension $n$ satisfying $\N_{2,2}$ and pick $p\in \p^r\backslash X$.
		Then
		\[
			\depth X_p=\min(\depth X, \dim\wt{Z}_2(p,X)+2)
		\]
		if $H^i(\mathcal{O}_X(j))=0$ for all $1\leq i\leq n-1$ and $j\leq -i$.
	\end{proposition}
	
	\section{Set-theoretic equations of $S_{k,d}(X)$}\label{sec:Skd_eqns}
	First, we extend the notions $\Sigma_d(X)$, $\Sigma^*_d(X)$, $S_d(X)$, and $S^*_d(X)$ defined in \Cref{def_sd*}. The following notions agree with \Cref{def_sd*} when $k=1$.
	
	\begin{definition}\label{def_skd}
		For a projective scheme $X\subseteq \p^r$, $k\geq 1$, and $d\geq 2$, define
		{\small
		\begin{itemize}[leftmargin=10pt]
			\item $\Sigma_{k,d}(X)= \left\{L\in\g(k,r)\,\middle\vert\,
			\begin{aligned}
				&L\cap X\textnormal{ scheme-theoretically contains a hypersurface}\\
				&\textnormal{of degree $\geq d$ in $L$}
			\end{aligned}
			\right\}_{\textstyle\raisebox{4pt},}$\\
			\item $\Sigma^*_{k,d}(X) = \left\{L\in\g(k,r)\,\middle\vert\,
			\begin{aligned}
				&L\cap X\textnormal{ scheme-theoretically contains a hypersurface}\\
				&\textnormal{of degree $\geq d$ in $L$ and }\dim(L\cap X_{\textnormal{sing}})\leq k-2
			\end{aligned}
			\right\}_{\textstyle\raisebox{4pt},}$\\
			\item $S_{k,d}(X)=\bigcup\limits_{L\in \Sigma_{k,d}(X)}L\subseteq\p^r,\text{ and }$\\
			\item $S^*_{k,d}(X)=\bigcup\limits_{L\in \Sigma^*_{k,d}(X)}L\subseteq\p^r.$
		\end{itemize}
		}
	\end{definition}

	We also denote
	\begin{align*}
			\Sigma_{k,\infty}(X)&=\{L\in\g(k,r)\mid L\subseteq X\}\text{ and}\\
			S_{k,\infty}(X)&=\bigcup\limits_{L\in \Sigma_{k,\infty}(X)}L\subseteq\p^r.
	\end{align*}
	By definition, we have $\Sigma_{k,\infty}(X)\subseteq \Sigma_{k,d}(X)$ and $S_{k,\infty}(X)\subseteq S_{k,d}(X)$.
	
	\begin{remark}\label{rem_skd=skd*_if_sing_small}
		If $\dim X_{\text{sing}}\leq k-2$, in particular, if $X$ is smooth, then \[\Sigma_{k,d}(X)=\Sigma^*_{k,d}(X)\text{ and }S_{k,d}(X)=S^*_{k,d}(X).\]
	\end{remark}

	\begin{proposition}\label{prop:Skd_closed}
		The loci $\Sigma_{k,d}(X)\subseteq \g(k,r)$ and $S_{k,d}(X)\subseteq\p^r$ are closed.
	\end{proposition}
	\begin{proof}
		Note that $\Sigma_{k,d}(X)=\Sigma_{k,\infty}(X)$ when $d>\maxdeg X$ where $\maxdeg(X)$ denotes the maximal degree in a minimal generating set of $I_X$.
		Hence $\Sigma_{k,d}(X)$ is a finite union of the subsets of the form
		\[
			\Gamma_\delta=\left\{L\in\g(k,r)\,\middle\vert\,
				\begin{aligned}
					L\cap X\text{ contains a hypersurface of degree $\delta$ in $L$}
				\end{aligned}
				\right\}.
		\]
		Let $\mathcal{S}\subseteq \g(k,r)\times \c^{r+1}$ be the universal subbundle over $\g(k,r)$.
		Then the projective bundle $$\p(\Sym^\delta(\mathcal{S}^{\vee}))\to\g(k,r)$$ is identified with the variety of hypersurfaces of degree $\delta$ in $k$-planes in $\p^r$ (see \cite[Example 14.7.12]{MR1644323}).
		As the relation $H\subseteq X$ gives a closed condition for $H\in \p(\Sym^\delta(\mathcal{S}^{\vee}))$, it holds that $\Gamma_\delta$ is closed.
		Hence $\Sigma_{k,d}(X)$ is closed.
		Let $$\Gamma_{\g(k,r)}=\{(L,p)\in\g(k,r)\times \p^r\mid p\in L\}\subseteq \g(k,r)\times \p^r$$ be the incidence correspondence and denote the projections as follows.
		\[
		\begin{tikzcd}
			\Gamma_{\g(k,r)} \ar[r,"\pi_2"]\ar[d,"\pi_1"] & \p^r \\
			\g(k,r)
		\end{tikzcd}
		\]
		As $S_{k,d}(X)=\pi_2(\pi_1^{-1}(\Sigma_{k,d}(X)))$, the result follows.
	\end{proof}
	Note that since $\dim(L\cap X_{\textnormal{sing}})\leq k-2$ is an open condition on $\g(k,r)$, the subset $\Sigma^*_{k,d}(X)\subseteq \g(k,r)$ is locally closed and $S^*_{k,d}(X)\subseteq\p^r$ is constructible.
	
	Next, we give set-theoretic equations of $S_{k,d}(X)\cup X$ for projective schemes $X\subseteq\p^r$ satisfying $\N_{d,2}$ (or $\K_d$). In particular, it enables us to determine the locus of $S_{1,d}(X)\cup X$ and compute its dimension with computer algebra systems.
	
	\begin{proposition}\label{prop_Nd2_implies_Kd}
		Let $X\subseteq\p^r$ be a projective scheme that is scheme-theoretically defined by homogeneous polynomials satisfying $\N_{d,2}$. Then $X$ satisfies $\K_d$.
	\end{proposition}
	\begin{proof}
		Denote the polynomial ring of $\p^r$ as $S=\c[x_0,\dots,x_r]$.
		By the assumption, there are forms $F_0,\dots,F_s$ of degree at most $d$ that define $X$ scheme-theoretically.
		Denote $d_i:=\deg F_i$.
		We denote the set of $d$-forms
		\[
			\Gamma := \{\mathcal{M}F_i\mid 0\le i\le s,\ \mathcal{M}\text{ is a monomial of degree $d-d_i$}\}.
		\]
		Then $d$-forms in $\Gamma$ scheme-theoretically define $X$.
		Furthermore, these still satisfy $\N_{d,2}$.
		Indeed, for fixed $i$, the syzygies between the polynomials of the form $\mathcal{M}F_i$, where $\mathcal{M}$ runs over monomials of degree $d-d_i$, are generated by linear syzygies.
		Hence it is enough to consider the syzygies of the polynomials with different $F_i$'s.
		By the condition $\N_{d,2}$, syzygies of $F_0,\dots,F_s$ of degree $\ge d+1$ are generated by the syzygies of the form $(G_0,\dots,G_s)$ where $G_i$ is a $(d+1-d_i)$-form for each $i$.
		Then we represent $G_i$ as a sum of monomials of degree $d+1-d_i$.
		This allows us to write $G_i$ as an $S_1$-linear combination of monomials of degree $d-d_i$.
		Hence, syzygies of the $d$-forms in $\Gamma$ are generated by linear syzygies except the syzygies of degree $d$, i.e., those of the form $(c_1,\cdots,c_{|\Gamma|})$ where $c_i\in\c$ for each $i$.
		The latter part can be ignored by taking a minimal generating subset of $\Gamma$ to remove the redundant elements that are linearly dependent.
		Therefore, the trivial syzygies of such $d$-forms are generated by linear syzygies, i.e., $X$ satisfies $\K_d$.
	\end{proof}

	The following example shows the principle of the proof of \Cref{prop_Nd2_implies_Kd} explicitly.

	\begin{example}\label{rem_Nd2_implies_Kd}
		The rational quartic curve $X\subseteq\p^3$ has the following Betti table, hence it satisfies the property $\N_{3,2}$.
		\[\begin{array}{c|cccc}
			&0 &1 &2 &3 \\ \hline
			0&1 & & & \\
			1& &1 & & \\
			2& &3 &4 &1
		\end{array}\]
		Let $Q$ be the quadratic form and $C_1,C_2,C_3$ be the cubic forms that generate the saturated ideal of $X$. Then
		\begin{equation}\label{quadric_cubic}
			(x_0Q,x_1Q,x_2Q,x_3Q,C_1,C_2,C_3)
		\end{equation}
		scheme-theoretically defines $X$ where $x_0,\dots,x_3$ are the homogeneous coordinates of $\p^3$. Suppose $(Q',L_1,L_2,L_3)$ is a syzygy of $(Q,C_1,C_2,C_3)$, i.e. \[Q'Q+L_1C_1+L_2C_2+L_3C_3=0\]
		where $Q'$ is a quadratic form and $L_i$'s are linear forms.
		By representing \[Q'=x_0L_0'+x_1L_1'+x_2L_2'+x_3L_3'\]
		where $L_0',\dots,L_3'$ are linear forms, we can see the syzygies of (\ref{quadric_cubic}) are generated by linear syzygies. Hence (\ref{quadric_cubic}) satisfies $\K_3$.
	\end{example}
	Hence all results in this section for projective schemes satisfying $\K_d$ also apply to those satisfying $\N_{d,2}$. Now we give set-theoretic equations of $S_{k,d}(X)\cup X$.

	\begin{thmalphabetmaintext}[Set-theoretic equations of $S_{k,d}(X)\cup X$]\label{thm_equations_of_skd}
		Let $X\subseteq\p^r=\textup{Proj}\ \c[x_0,\dots,x_r]$ be a (possibly singular) projective scheme that is scheme-theoretically defined by $d$-forms $F_0,\dots,F_s$ satisfying $\N_{d,2}$ (or more generally $\K_d$). Let \[\{(L_{i0},\dots,L_{is})\}_{1\leq i\leq m}\] be a generating set of the linear syzygies of the $F_i$'s, and write
		\[
		L_{ij} =a^0_{ij}x_0+\cdots+a^r_{ij}x_r
		\]
		where $a^k_{ij}\in\c$.
		Let $M$ be the $m\times (r+1)$ matrix of $d$-forms as follows.
		\[M=\begin{bmatrix}
			\sum\limits^s_{j=0} a^0_{1j}F_j &\sum\limits^s_{j=0} a^1_{1j}F_j& \cdots & \sum\limits^s_{j=0} a^r_{1j}F_j\\
			\sum\limits^s_{j=0} a^0_{2j}F_j &\sum\limits^s_{j=0} a^1_{2j}F_j& \cdots & \sum\limits^s_{j=0} a^r_{2j}F_j\\
			\vdots & \vdots & \ddots & \vdots \\
			\sum\limits^s_{j=0} a^0_{mj}F_j &\sum\limits^s_{j=0} a^1_{mj}F_j& \cdots & \sum\limits^s_{j=0} a^r_{mj}F_j\\
		\end{bmatrix}\]
		Then $I_{r+1-k}(M)$ defines $S_{k,d}(X)\cup X$ set-theoretically for every $k\geq 1$.
		In particular, $I_r(M)$ defines $S_{1,d}(X)\cup X$ set-theoretically.
		Hence $S_{1,d}(X)=\p^r$ if and only if every $r$-minor of $M$ is zero.
	\end{thmalphabetmaintext}
	\begin{proof}
		We use the construction in Section \ref{sec_vermeire} so that $\varphi:\p^r\backslash X\to \p^s$ is the map corresponding to $[F_0,\dots,F_s]$, and $\Gamma_\varphi\subseteq\p^r\times\p^s$ is the closure of the graph of $\varphi$ which is isomorphic to $\bl_X\p^r$. Also, $E$ is the exceptional divisor of the blowing up.
		\[
		\begin{tikzcd}
			Bl_X\p^r=\Gamma_\varphi \ar[r,"\pi_2"]\ar[d,"\pi_1"] &\p^s \\
			\p^r \ar[ur,"\varphi",dotted,swap]
		\end{tikzcd}
		\]
		Then for $a=[a_0,\dots,a_s]\in\pi_2(\Gamma_\varphi)$, $\pi_1(\pi_2^{-1}(a))$ is either defined by the linear forms \[\{L_{10}a_0+\cdots +L_{1s}a_s,\dots,L_{m0}a_0+\cdots +L_{ms}a_s\}\] or contained in $X$ by the proof of \Cref{prop_fiber_of_lifted_kd_map} and \Cref{rem_fiber_of_lifted_kd_map_without_assumption}. The second case occurs if and only if $\pi_2^{-1}(a)$ is contained in $E$.
		For $p\in\p^r\backslash X$, $\pi_2^{-1}(\varphi(p))$ is not contained in $E$, hence the second case does not occur.
		Thus for $p\in\p^r\backslash X$, $\pi_1(\pi_2^{-1}(\varphi(p)))$ is defined by the linear forms 
		\begin{gather*}
			\left(\sum\limits^s_{j=0} a^0_{1j}F_j(p)\right)x_0+\left(\sum\limits^s_{j=0} a^1_{1j}F_j(p)\right)x_1+\cdots+\left(\sum\limits^s_{j=0} a^r_{1j}F_j(p)\right)x_r,\\
			\left(\sum\limits^s_{j=0} a^0_{2j}F_j(p)\right)x_0+\left(\sum\limits^s_{j=0} a^1_{2j}F_j(p)\right)x_1+\cdots+\left(\sum\limits^s_{j=0} a^r_{2j}F_j(p)\right)x_r,\\
			\vdots\\
			\left(\sum\limits^s_{j=0} a^0_{mj}F_j(p)\right)x_0+\left(\sum\limits^s_{j=0} a^1_{mj}F_j(p)\right)x_1+\cdots+\left(\sum\limits^s_{j=0} a^r_{mj}F_j(p)\right)x_r.
		\end{gather*}
		Thus $\dim\pi_1(\pi_2^{-1}(\varphi(p)))\geq k$ if and only if the rank of the matrix
		\[M(p)=\begin{bmatrix}
			\sum\limits^s_{j=0} a^0_{1j}F_j(p) &\sum\limits^s_{j=0} a^1_{1j}F_j(p)& \cdots & \sum\limits^s_{j=0} a^r_{1j}F_j(p)\\
			\sum\limits^s_{j=0} a^0_{2j}F_j(p) &\sum\limits^s_{j=0} a^1_{2j}F_j(p)& \cdots & \sum\limits^s_{j=0} a^r_{2j}F_j(p)\\
			\vdots & \vdots & \ddots & \vdots \\
			\sum\limits^s_{j=0} a^0_{mj}F_j(p) &\sum\limits^s_{j=0} a^1_{mj}F_j(p)& \cdots & \sum\limits^s_{j=0} a^r_{mj}F_j(p)\\
		\end{bmatrix}\]
		is less than or equal to $r-k$.
		Therefore \[S_{k,d}(X)\backslash X=\{p\in\p^r\backslash X\mid \rank M(p)\leq r-k\}\]
		by \Cref{rem_k-plane_maps_to_a_point} and \Cref{rem_fiber_of_lifted_kd_map_without_assumption} since $X$ is scheme-theoretically defined by $d$-forms.
		Also, $\rank M(p)=0$ for $p\in X$ and $\rank M(p)\geq r+1-k$ for $p\notin S_{k,d}(X)\cup X$. Hence \[S_{k,d}(X)\cup X=\{p\in\p^r\mid\rank M(p)\leq r-k\}=V(I_{r+1-k}(M)).\qedhere\]
	\end{proof}
	
	\begin{example}\label{example:RNC}
		Let $X\subseteq\p^4=\proj\c[x_0,\dots,x_4]$ be the rational normal curve defined by $2$-minors of the following matrix
		\[\begin{bmatrix}
			x_0 & x_1 & x_2 & x_3\\
			x_1 & x_2 & x_3 & x_4
		\end{bmatrix}_{\textstyle\raisebox{4pt}.}\]
		The homogeneous ideal of $X$ is $I_X=(F_0,\dots,F_5)$ where
		\begin{alignat*}{6}
			F_0 &= -x_1^2+x_0x_2,\quad
			&F_1 &= -x_1x_2+x_0x_3,\quad
			&F_2 &= -x_2^2+x_1x_3,\\
			F_3 &= -x_1x_3+x_0x_4,\quad
			&F_4 &= -x_2x_3+x_1x_4,\quad
			&F_5 &= -x_3^2+x_2x_4,
		\end{alignat*}
		and it has the following Betti table.
		\[\begin{array}{c|cccc}
			&0 &1 &2 &3 \\ \hline
			0&1 & & & \\
			1& &6 &8 &3 \\
		\end{array}\]
		The linear syzygies of $F_0,\dots,F_5$ are
		\begin{gather*}
			(x_3,-x_2,x_1,0,0,0),\\
			(x_2,-x_1,x_0,0,0,0),\\
			(0,x_4,-x_3,-x_3,x_2,0),\\
			(x_4,0,0,-x_2,x_1,0),\\
			(x_3,0,0,-x_1,x_0,0),\\
			(0,0,x_4,0,-x_3,x_2),\\
			(0,x_4,0,-x_3,0,x_1),\\
			(0,x_3,0,-x_2,0,x_0).
		\end{gather*}
		Hence we get the matrix
		{\footnotesize
		\[
			M= \begin{bmatrix}
				0   & F_2  & -F_1 & F_0      & 0   \\
				F_2 & -F_1 & F_0  & 0        & 0   \\
				0   & 0    & F_4  & -F_2-F_3 & F_1 \\
				0   & F_4  & -F_3 & 0        & F_0 \\
				F_4 & -F_3 & 0    & F_0      & 0   \\
				0   & 0    & F_5  & -F_4     & F_2 \\
				0   & F_5  & 0    & -F_3     & F_1 \\
				F_5 & 0    & -F_3 & F_1      & 0
			\end{bmatrix}
		\]}		
		{\scriptsize
		\[
			=\begin{bmatrix}
				0              & -x_2^2+x_1x_3  & x_1x_2-x_0x_3  & -x_1^2+x_0x_2  & 0              \\
				-x_2^2+x_1x_3  & x_1x_2-x_0x_3  & -x_1^2+x_0x_2  & 0              & 0              \\
				0              & 0              & -x_2x_3+x_1x_4 & x_2^2-x_0x_4   & -x_1x_2+x_0x_3 \\
				0              & -x_2x_3+x_1x_4 & x_1x_3-x_0x_4  & 0              & -x_1^2+x_0x_2  \\
				-x_2x_3+x_1x_4 & x_1x_3-x_0x_4  & 0              & -x_1^2+x_0x_2  & 0              \\
				0              & 0              & -x_3^2+x_2x_4  & x_2x_3-x_1x_4  & -x_2^2+x_1x_3  \\
				0              & -x_3^2+x_2x_4  & 0              & x_1x_3-x_0x_4  & -x_1x_2+x_0x_3 \\
				-x_3^2+x_2x_4  & 0              & x_1x_3-x_0x_4  & -x_1x_2+x_0x_3 & 0
			\end{bmatrix}
		\]}
		\noindent
		as in \Cref{thm_equations_of_skd}.
		Then $I_4(M)$ set-theoretically defines $$S_{1,2}(X)\cup X=\sigma_2X.$$
		Using Macaulay2, one can see \[V(I_4(M))=V(x_2^3-2x_1x_2x_3+x_0x_3^2+x_1^2x_4-x_0x_2x_4)\] which is indeed the locus of $\sigma_2X$.
	\end{example}
	
	\begin{example}\label{example:RQC}
		Let $X\subseteq\p^3=\proj\c[x_0,\dots,x_3]$ be the rational quartic curve that is the image of
		\begin{align*}
			\proj\c[s,t] &\to \proj\c[x_0,\dots,x_3]\\
			[s,t] &\mapsto [s^4,s^3t,st^3,t^4].
		\end{align*}
		The homogeneous ideal of $X$ is $I_X=(Q,C_1,C_2,C_3)$ where
		\begin{align*}
			Q &= x_1x_2-x_0x_3,\\
			C_1 &= x_2^3-x_1x_3^2,\\
			C_2 &= x_0x_2^2-x_1^2x_3,\\
			C_3 &= x_1^3-x_0^2x_2,
		\end{align*}
		and it has the following Betti table.
		\begin{gather*}
			\begin{array}{c|cccc}
				&0 &1 &2 &3 \\ \hline
				0&1 & & & \\
				1& &1 & & \\
				2& &3 &4 &1
			\end{array}\\
			\Small{\text{The Betti table of the saturated ideal $I_X$}}
		\end{gather*}
		Although it satisfies $\N_{3,2}$, the homogeneous ideal $I_X$ is not generated by only cubic forms.
		Hence we replace $Q$ by $(x_0Q,x_1Q,x_2Q,x_3Q)$ as in \Cref{rem_Nd2_implies_Kd}, and we get a desaturated ideal
		\[
			J:=(x_0Q,x_1Q,x_2Q,x_3Q,C_1,C_2,C_3)
		\]
		satisfying \textbf{K$_3$} that still defines $X$ scheme-theoretically with the following Betti table.
		\begin{gather*}
			\begin{array}{c|ccccc}
				&0 &1 &2 &3 &4 \\ \hline
				0&1 & & & & \\
				1& & & & & \\
				2& &7 &10 &5 &1
			\end{array}\\
			\Small{\text{The Betti table of a desaturated ideal $J$}}
		\end{gather*}
		Then we get a $10\times 4$ matrix $M$ of cubic forms whose $3$-minors define $S_{1,3}(X)\cup X$ set-theoretically. Using Macaulay2, one can see \[V(I_3(M))=V(x_1x_2-x_0x_3).\] Hence the locus of trisecant lines of $X$ is a quadric hypersurface.
	\end{example}

	\begin{example}\label{example:6pts}
		Let $X\subseteq\p^2=\proj \c[x,y,z]$ be the set of six points
		\[\{[0,0,1], [1,0,1], [2,0,1], [0,1,1], [1,1,1], [0,2,1]\}\] in the projective plane. In the affine plane $\mathbb{A}^2$ corresponding to $z\neq 0$, $X$ consists of the following six points.
		\begin{center}
			\includegraphics{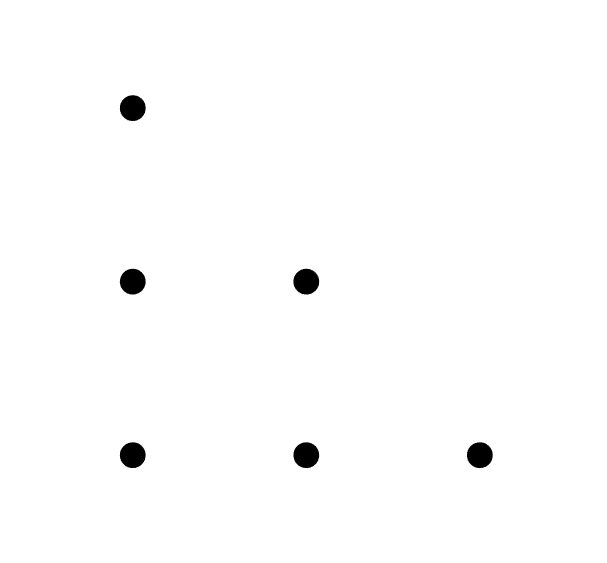}\\
			\Small{\text{Six points $X\subseteq\mathbb{A}^2\subseteq\p^2$}}
		\end{center}
		Using Macaulay2, one can see the homogeneous ideal of $X$ is
		\[I_X=(y^3-3y^2z+2yz^2, xy^2-xyz, x^2y-xyz, x^3-3x^2z+2xz^2),\]
		and the Betti table is given as follows.
		\[\begin{array}{c|ccc}
			&0 &1 &2 \\ \hline
			0&1 & & \\
			1& & & \\
			2& &4 &3 
		\end{array}\]
		Then we get a $3\times 3$ matrix $M$ whose $2$-minors define $S_{1,3}(X)\cup X$ set-theoretically.
		Also, one can see \[V(I_2(M))=V(x^2y+xy^2-2xyz)=V(x)\cup V(y)\cup V(x+y-2z)\] using Macaulay2.
		Indeed, the lines $V(x)$, $V(y)$, $V(x+y-2)$ on $\mathbb{A}^2$ are the trisecant lines of $X$ as follows.
		\vspace{5pt}
		\begin{center}
			\includegraphics{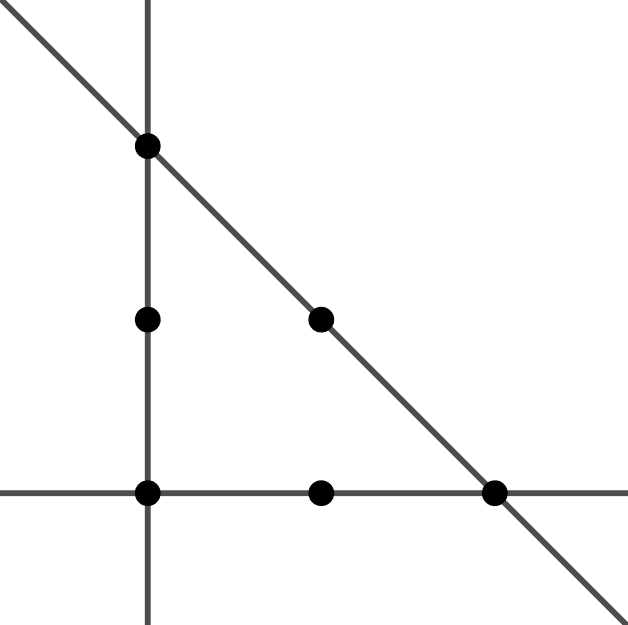}\\
			\Small{\text{Trisecant lines of $X\subseteq\mathbb{A}^2\subseteq\p^2$}}
		\end{center}
	\end{example}
	
	\begin{example}\label{example:curvilinear}
		Let $X\subseteq\p^2=\proj\c[x,y,z]$ be the union of
		\begin{itemize}
			\item the triple point at $[0,0,1]$ defined by $I_1=(x^3,y)$ and
			\item the line defined by $I_2=(x-z)$.
		\end{itemize}
		In the affine plane $\mathbb{A}^2$ corresponding to $z\neq 0$, it is the union of a triple point at the origin curvilinear along the $x$-axis and the line $V(x-1)$.
		\begin{center}
			\includegraphics{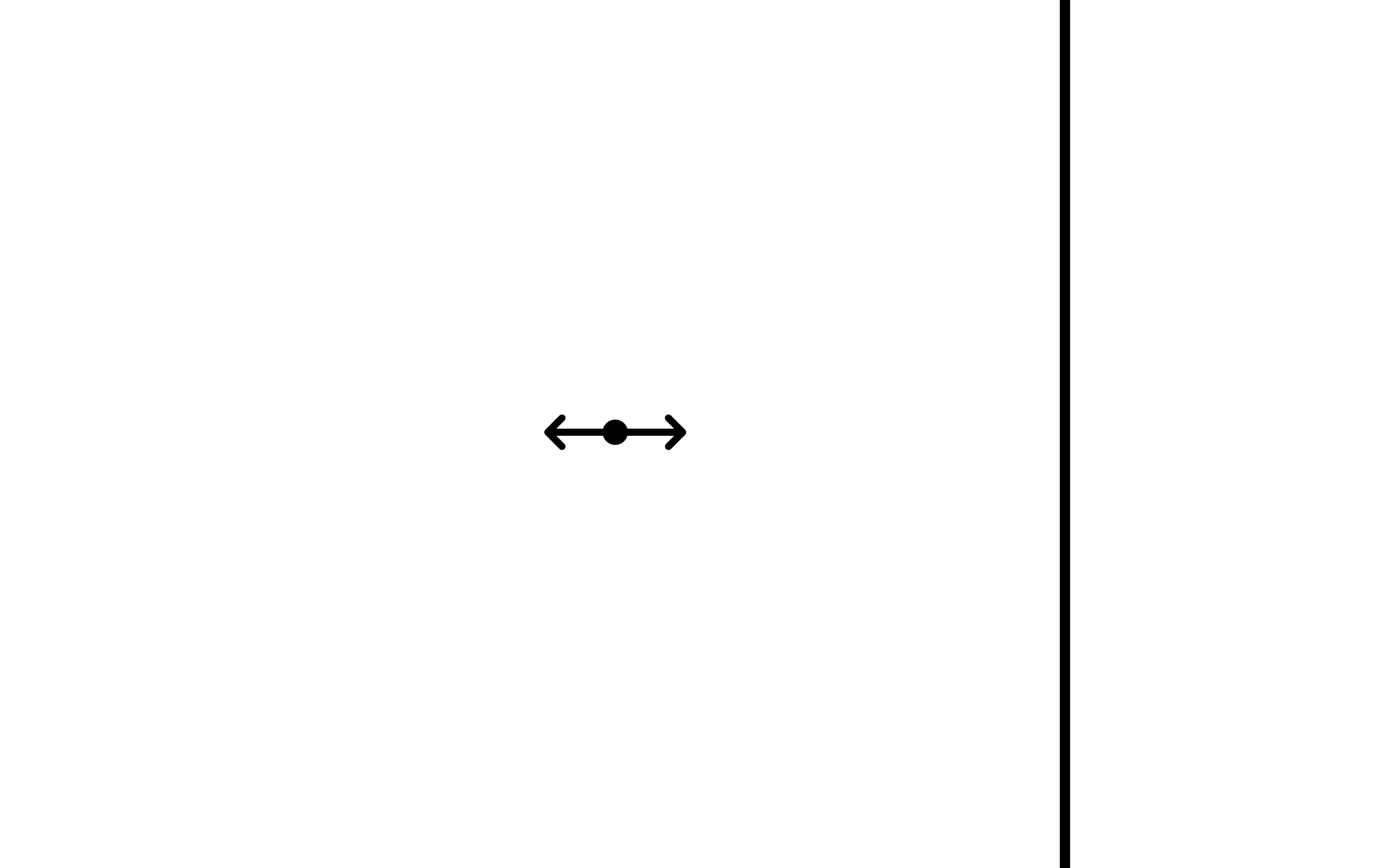}\\
			\Small{\text{Curvilinear triple point at the origin of $\mathbb{A}^2$ and the line $V(x-1)\subseteq\mathbb{A}^2$}}
		\end{center}
		The homogeneous ideal of $X$ is
		\[
		I_X=(xy-yz,x^4-x^3z),
		\]
		and has the following Betti table.
		\begin{gather*}
			\begin{array}{c|ccc}
				&0 &1 &2 \\ \hline
				0&1 & & \\
				1& &1 & \\
				2& & & \\
				3& &1 &1 
			\end{array}\\
			\Small{\text{The Betti table of the saturated ideal $I_X$}}
		\end{gather*}
		By replacing $xy-yz$ by $(xy-yz)\cdot(x,y,z)^2$, we get the ideal
		{\Small
			\[J=(xyz^2-yz^3,xy^2z-y^2z^2,x^2yz-yz^3,xy^3-y^3z,x^2y^2-y^2z^2,x^3y-yz^3,x^4-x^3z)\]}
		\noindent
		that still defines $X$ scheme-theoretically with the following Betti table.
		\begin{gather*}
			\begin{array}{c|cccc}
				&0 &1 &2 &3 \\ \hline
				0&1 & & & \\
				1& & & & \\
				2& & & & \\
				3& &7 &9 &3 
			\end{array}\\
			\Small{\text{The Betti table of a desaturated ideal $J$}}
		\end{gather*}
		Then we get a $9\times 3$ matrix $M$ whose $2$-minors define $S_{1,4}(X)\cup X$ set-theoretically.
		Using Macaulay2, we get
		\[
			V(I_{2}(M))=V(xy-yz)=V(y)\cup V(x-z).
		\]
		Indeed, $V(y)\subseteq\mathbb{A}^2$ is a $4$-secant line of $X$, and $V(x-1)\subseteq\mathbb{A}^2$ is a line contained in $X$.
		\begin{center}
			\includegraphics{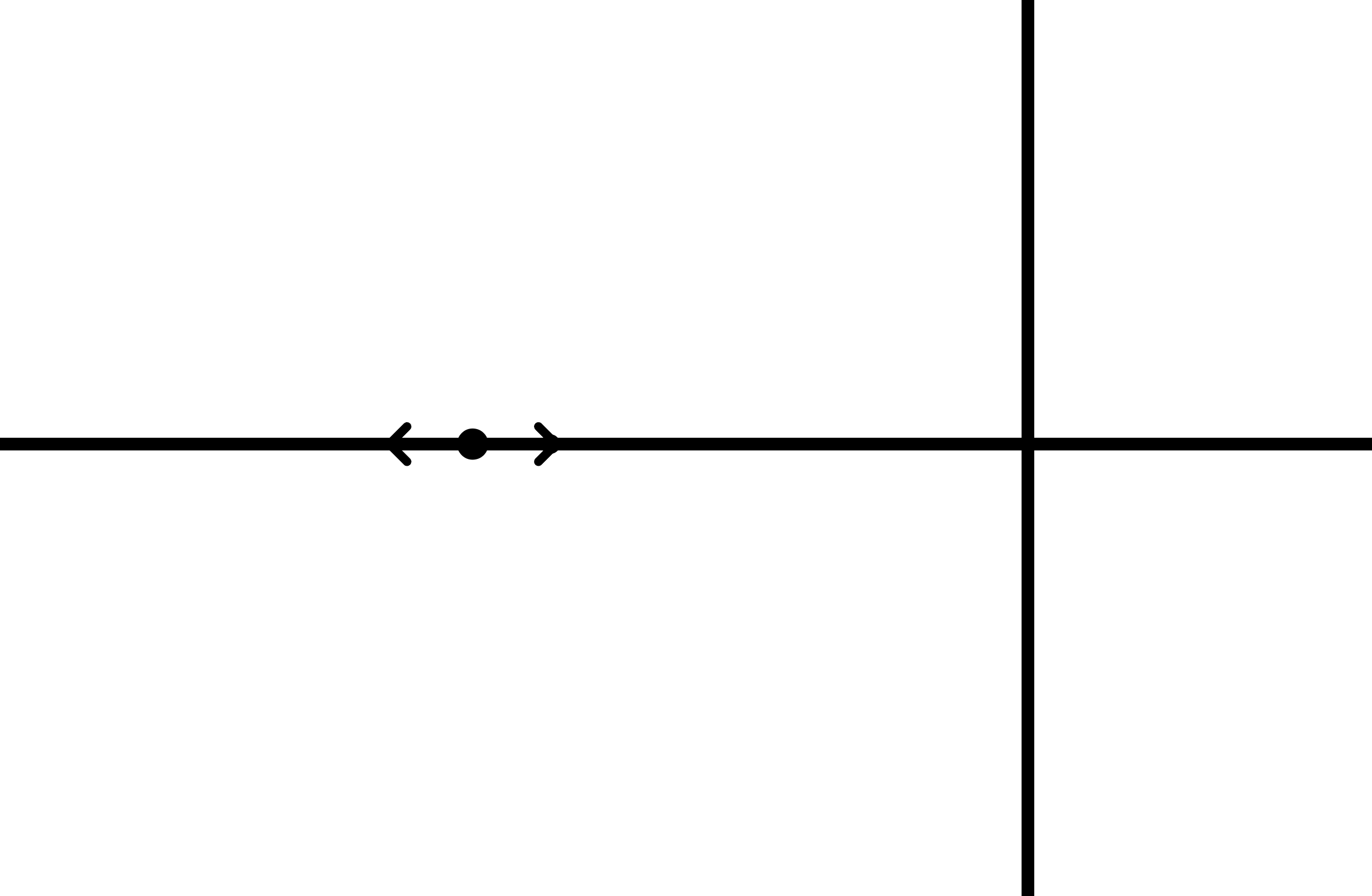}\\
			\Small{\text{$4$-secant lines of $X\subseteq\mathbb{A}^2\subseteq\p^2$}}
		\end{center}
	\end{example}
	
	Using the notion of $S_{k,d}(X)$, we give a filtration.
	
	\begin{proposition}\label{prop_represent_higher_secant_variety}
		Let $X\subseteq\p^r$ be a nondegenerate projective variety and let $q\ge 1$. Then we have
		\begin{equation}\label{eqn_higher_secant_representation}
			\sigma_{q+1}X\subseteq S_{q,q+1}(\sigma_{q}X)\subseteq S_{q-1,q+1}(\sigma_{q}X)\subseteq\cdots\subseteq S_{1,q+1}(\sigma_{q}X).
		\end{equation}
	\end{proposition}
	\begin{proof}
		Let $L$ be a $(q+1)$-secant $q$-plane to $X$. Let $p_1,\dots,p_{q+1}$ be the points spanning $L$. Then whenever we pick $q$ points among $p_i$, we get a $q$-secant $(q-1)$-plane to $X$ which is contained in $\sigma_qX$. Hence $L\cap \sigma_qX$ contains a hypersurface of degree $q+1$ of $L$, and we get $\sigma_{q+1}X\subseteq S_{q,q+1}(\sigma_qX)$ since $S_{q,q+1}(\sigma_qX)$ is closed (\Cref{prop:Skd_closed}). Other containments are obvious.
	\end{proof}
	
	It is well known that $(I_{\sigma_qX})_q=0$ \cite[Corollary 2.3]{MR1822836}.
	Hence this filtration gives the following corollary of \Cref{thm_equations_of_skd}.
	
	\setcounter{thmalphabetmaintext}{2}
	\begin{coralphabetmaintext}\label{cor_equations_of_higher_secant}
		Let $X\subseteq\p^r$ be a nondegenerate projective variety such that $\sigma_qX\subsetneq\p^r$ satisfies the property $\N_{q+1,2}$ where $q\ge 1$.
		Define $M$ as in \Cref{intro_set-theoretic_eq} for the minimal generators of $I_{\sigma_qX}$. Then it holds that
		\[
			I_{r-q+1}(M)\subseteq I_{\sigma_{q+1}X}.
		\]
	\end{coralphabetmaintext}
	
	\section{The dimension bound}\label{sec_dimension_bound}
	
	In this section, we give an upper bound of $\dim S^*_{k,d}(X)$ which can be regarded as an extension of the $d$-secant lemma to $k$-planes.
	
	Recall that a constructible subset $A$ of a Zariski topological space contains a dense open subset of $\overline{A}$, and $\dim A=\dim \overline{A}$.
	
	\begin{lemma}\label{lem_closure_in_grassmannian}
		Let $\Sigma_0$ be a constructible subset of $\g(k,r)$, and let $\Sigma:=\overline{\Sigma_0}\subseteq\g(k,r)$. Then $\bigcup\limits_{L\in\Sigma_0}L$ is a constructible subset of $\p^r$, and $\bigcup\limits_{L\in\Sigma}L=\overline{\bigcup\limits_{L\in\Sigma_0}L}$. In particular, $\dim\bigcup\limits_{L\in\Sigma}L=\dim\bigcup\limits_{L\in\Sigma_0}L$.
	\end{lemma}
	\begin{proof}
		Let $\Gamma_{\Sigma}:=\{(L,x)\in\Sigma\times\p^r\mid x\in L\}$ be the incidence correspondence over $\Sigma$.
		\[
		\begin{tikzcd}
			\Gamma_{\Sigma} \ar[r,"\pi_2"]\ar[d,"\pi_1"] &\p^r\\
			\Sigma
		\end{tikzcd}
		\]
		Then $\bigcup\limits_{L\in\Sigma_0}L=\pi_2(\pi_1^{-1}(\Sigma_0))$ is a constructible subset of $\p^r$. As $\Sigma_0$ is constructible, it contains a dense open subset $U$ of $\Sigma$. Then $\pi_1^{-1}(U)$ is a dense open subset of $\Gamma_\Sigma$, hence $\pi_2(\pi_1^{-1}(U))$ is a dense subset of $\pi_2(\Gamma_\Sigma)=\bigcup\limits_{L\in\Sigma}L$. This is because taking the closure and taking the image commute for any closed morphism.
		Thus $\bigcup\limits_{L\in\Sigma_0}L=\pi_2(\pi_1^{-1}(\Sigma_0))\supseteq\pi_2(\pi_1^{-1}(U))$ is a dense subset of $\bigcup\limits_{L\in\Sigma}L$, i.e. $\overline{\bigcup\limits_{L\in\Sigma_0}L}=\bigcup\limits_{L\in\Sigma}L$. Also, they have the same dimension as $\bigcup\limits_{L\in\Sigma_0}L$ is constructible.
	\end{proof}
	
	Now we prove that without changing the dimension of the locus of $k$-planes contained in the base, the base can be assumed to be a closed subvariety of $\g(k,r)$, and furthermore, its dimension can be assumed to be minimal in the following sense.
	
	\begin{lemma}\label{lem_reduce_base_in_grassmannian}
		Let $\Sigma_0$ be a nonempty constructible subset of $\g(k,r)$, and $ S_0:=\bigcup\limits_{L\in\Sigma_0}L\subseteq\p^r$. Let $\dim S_0=m$. Then there is a closed subvariety $\Sigma$ of $\g(k,r)$ of dimension $m-k$ such that the general element of $\Sigma$ is contained in $\Sigma_0$, and $ S:=\bigcup\limits_{L\in\Sigma}L\subseteq\p^r$ has dimension $m$.
	\end{lemma}
	\begin{proof}
		Note that there is an irreducible component $\Sigma_1$ of $\Sigma_0$ such that $\bigcup\limits_{L\in\Sigma_1}L$ has dimension $m$ in $\p^r$.
		Obviously $\Sigma_1$ is a constructible subset of $\g(k,r)$.	
		Let $\Sigma$ be the closure of $\Sigma_1$ in $\g(k,r)$, and $$\Gamma_\Sigma:=\{(L,x)\in\Sigma\times\p^r\mid x\in L\}$$ be the incidence correspondence over $\Sigma$.
		By \Cref{lem_closure_in_grassmannian}, $$S:=\textnormal{im}(\Gamma_{\Sigma}\to\p^r)=\bigcup\limits_{L\in\Sigma}L$$ has dimension $m$.
		
		Now suppose $\Sigma$ has dimension greater than $m-k$. Then $\Gamma_{\Sigma}$ has dimension greater than $m$, hence the general fiber of $\Gamma_{\Sigma}\to  S$ has dimension $\geq 1$.
		\[
		\begin{tikzcd}
			\Gamma_{\Sigma} \ar[r,"\pi_2"]\ar[d,"\pi_1"] & S \\
			\Sigma
		\end{tikzcd}
		\]
		Then for the general point $p\in S$, \[\pi_1(\pi_2^{-1}(p))=\{L\in\Sigma\mid p\in L\}\]
		also has dimension $\geq 1$. Hence $\pi_1(\pi_2^{-1}(p))$ meets any hyperplane of $\p^{{r+1\choose k+1}-1}$ which denotes the ambient projective space of $\g(k,r)$.
		Let $H$ be the general hyperplane of $\p^{{r+1\choose k+1}-1}$. 
		Note that if $\dim \Sigma=1$, then $m=k$ so that $S$ is a $k$-plane and $\Sigma$ is a point, a contradiction.
		Thus we may assume $\dim \Sigma\ge 2$.
		Then $H\cap\Sigma$ is a closed subvariety of $\g(k,r)$ by Bertini's theorem.
		Let $\Gamma_{H\cap \Sigma}$ be the incidence variety over $H\cap \Sigma$.
		Then $\Gamma_{H\cap \Sigma}$ still surjects into $ S$ since we have seen $\pi_1(\pi_2^{-1}(p))$ meets $H$ for general $p\in  S$.
		By the constructibility of $\Sigma_1$, the general element of $H\cap \Sigma$ is contained in $\Sigma_1$, hence also in $\Sigma_0$.
		Therefore we can reduce the dimension of $\Sigma\subseteq\p^{{r+1\choose k+1}-1}$ by taking the general hyperplane sections until the dimension of $\Sigma$ becomes $m-k$.
	\end{proof}
	
	Now we prove the dimension bound stated in \Cref{intro_dimension_bound} in the following consecutive theorems.
	First, we show that if the dimension $r$ of the ambient projective space is large enough, then $\dim S^*_{k,d}(X)$ is less than $r$.
	The proof of the following theorem follows the argument in \cite[Theorem A.3]{MR3705293} which considers the case $k=1$.
	We extend it to the higher dimensional cases.
	\begin{theorem}\label{thm_dim_skd<=r-1}
		Let $X\subseteq\p^r$ be a projective variety of dimension $n$ and let $1\le k\le n+1$ and $d\ge 2$.
		If $r>n+1+\dfrac{n+1-k}{d-1}$, then $\dim S^*_{k,d}(X)\leq r-1$. 
	\end{theorem}
	\begin{proof}
		Suppose $\dim S^*_{k,d}(X)=r$.
		By \Cref{lem_reduce_base_in_grassmannian}, there is a closed subvariety $\Sigma_0\subseteq\g(k,r)$ of dimension $r-k$ such that the general element of $\Sigma_0$ is contained in $\Sigma^*_{k,d}(X)$, and the incidence variety $$\Gamma_{\Sigma_0}\subseteq\g(k,r)\times\p^r$$ over $\Sigma_0$ maps to $\p^r$ surjectively. Let $\rho:\Sigma\to \Sigma_0$ be a surjective desingularization of $\Sigma_0$, and $\pi:\Gamma_{\Sigma}\to \Sigma$ be the pullback of $\Gamma_{\Sigma_0}\to \Sigma_0$ by $\rho:\Sigma\to \Sigma_0$ as follows.
		\[
		\begin{tikzcd}
			\Gamma_\Sigma \ar[r,swap,"\eta"]\ar[d,"\pi"]\ar[rr,bend left=18,"\mu"]\ar[rd,phantom,"\square"] &\Gamma_{\Sigma_0} \ar[r,swap,"\mu_0"]\ar[d,"\pi_0"] &\p^r \\
			\Sigma \ar[r,swap,"\rho"] &\Sigma_0
		\end{tikzcd}
		\]
		Then $\Sigma$ is a smooth projective variety of dimension $r-k$, and $\Gamma_\Sigma$ is a smooth projective variety of dimension $r$. As $\eta$ is a base change of $\rho$, it is surjective. Hence $\mu$ is surjective and generically finite.
		
		Let $\mu^{-1}(X)$ be the scheme-theoretic inverse image of $X$. 
		Then $$\dim \mu^{-1}(X)\leq r-1$$ by the surjectivity of $\mu$. Suppose $\dim \mu^{-1}(X)\leq r-2$. Let $L_0$ be the general fiber of $\pi$. Then $L_0$ must meet $\mu^{-1}(X)$ on the locus of dimension $\geq k-1$ as the general element of $\Sigma$ is contained in $\Sigma_{k,d}^*(X)$. Hence the general fiber of $\pi|_{\mu^{-1}(X)}$ has dimension $\geq k-1$, and with the fact that $\pi|_{\mu^{-1}(X)}$ is surjective, we get a contradiction to $\dim\Sigma=r-k$.
		Let $D$ be the largest effective divisor of $\Gamma_\Sigma$ contained in $\mu^{-1}(X)$, and let $Z_1,\dots,Z_s$ be the irreducible components and embedded components of $\mu^{-1}(X)$ of dimension $\leq r-2$.
		Write $D$ as a sum of effective divisors $$D=E+E'+E''$$ such that the prime divisors of $E$ dominate $\pi$ and $\mu(E)\cap X_{\textnormal{sm}}\neq\varnothing$, the prime divisors of $E'$ dominate $\pi$ and $\mu(E')\cap X_{\textnormal{sm}}=\varnothing$, and the prime divisors of $E''$ do not dominate $\pi$.
		
		Let $L_0$ be the general fiber of $\pi$. Then $E''$ does not intersect with $L_0$. Also if $(L_0,x)\in$ $E'\cap L_0$, then $x\in X_{\textnormal{sing}}$. Hence $$\dim (E'\cap L_0)\leq k-2.$$ Also $\dim (Z_i\cap L_0)\leq k-2$ since if $Z_i$ meets the general fiber $L_0$, then $Z_i$ surjects into $\Sigma$ where $\dim Z_i\leq r-2$, $\dim \Sigma=r-k$.
		Hence the scheme-theoretic intersection $E\cap L_0$ scheme-theoretically contains a degree $\geq d$ hypersurface $H$ of $L_0$ possibly except some locus of dimension $\leq k-2$. However, as $E\cap L_0$ is closed, it should contain the entire $H$ scheme-theoretically since $H$ is Cohen--Macaulay so that it does not have embedded points. In particular, we have $H\subseteq E$ as schemes. Let $c:=\codim X$ and $(c-1)H$ be the divisor of $L_0$. Then we consider $(c-1)H$ as a closed subscheme of $\Gamma_\Sigma$. Let $(c-1)E$ be the divisor of $\Gamma_\Sigma$. Then $(c-1)H\subseteq (c-1)E$ as schemes. Thus $$(c-1)H\subseteq (c-1)E\cap L_0$$ as schemes.
		
		Now write $E=\sum\limits_{i}a_iE_i$ where $a_i\in\mathbb{N}$, and $E_i$'s are distinct prime divisors of $E$.
		Let $R$ be the ramification divisor of $\mu$. Then by \Cref{lem_order_of_ramification_divisor}, $$\textnormal{ord}_{E_i}R\geq (c-1)a_i,$$ hence $R\geq (c-1)E$.
		Then $$R\cap L_0\supseteq (c-1)E\cap L_0\supseteq (c-1)H$$ as schemes.
		By \Cref{lem_focal_locus_is_hypersurface_of_degree_r-k}, $R\cap L_0$ is a degree $r-k$ hypersurface of $L_0$ since the general fiber $L_0$ is not contained in $R$.
		As $(c-1)H$ is a degree $\geq (c-1)d$ hypersurface of $L_0$, we get $r-k\geq (c-1)d$. This implies $$r\leq n+1+\dfrac{n+1-k}{d-1}$$ which is a contradiction by the assumption.	
	\end{proof}
	
	Then we show a general projection does not reduce the dimension of $S^*_{k,d}$.
	
	\begin{lemma}\label{lem_general_projection_preserves_dim_skd}
		Let $X\subseteq\p^r$ be a projective variety of dimension $n$ and let $1\le k\le n+1$ and $d\ge 2$.
		Suppose $m:=\dim S^*_{k,d}(X)\geq n+1$.
		If $r\geq n+k+2$, then for the general projection $\pi_q:\p^r\backslash\{q\}\to \p^{r-1}$, we have 
		\begin{equation*}
			\dim S^*_{k,d}(\pi_q(X))\begin{cases}
				\geq m &\text{if $m<r$}\\
				=r-1 &\text{if $m=r$}.
			\end{cases}
		\end{equation*}	
		Here $\pi_q(X)$ is considered as a projective variety in $\p^{r-1}$, hence $S^*_{k,d}(\pi_q(X))$ is the union of some $k$-planes in $\p^{r-1}$.
	\end{lemma}
	\begin{proof}
		By \Cref{lem_reduce_base_in_grassmannian}, there is a closed subvariety $\Sigma\subseteq \g(k,r)$ of dimension $m-k$ such that the general element of $\Sigma$ is contained in $\Sigma^*_{k,d}(X)$, and $ S:=\bigcup\limits_{L\in \Sigma}L$ has dimension $m$.
		
		\[
		\begin{tikzcd}
			\Gamma_{\Sigma} \ar[r,"\pi_2"]\ar[d,"\pi_1"] & S \\
			\Sigma
		\end{tikzcd}
		\]
		
		Let $\Sigma_0$ be the dense open subset of $\Sigma$ such that $\Sigma_0\subseteq \Sigma^*_{k,d}(X)$. Then $\bigcup\limits_{L\in\Sigma_0}L$ also has dimension $m$ by \Cref{lem_closure_in_grassmannian}.
		Pick a general element $L_0\in \Sigma_0$ so that $L_0\nsubseteq X$, which is ensured by the condition $m\geq n+1$.
		Then
		\[
			\dim (L_0\cap X_{\text{sing}})\leq k-2,
		\]
		and $L_0\cap X$ contains a degree $\geq d$ hypersurface of $L_0$. In particular, the union of the irreducible components of $L_0\cap X$ of dimension $k-1$ is a degree $\geq d$ hypersurface of $L_0$. Let $H_1,\dots ,H_s$ be the irreducible components of $L_0\cap X$ of dimension $k-1$. As $\dim (L_0\cap X_{\text{sing}})\leq k-2$, each $H_i$ contains a smooth point $p_i$ of $X$.
		Let $J(L_0,X)$ be the join of $L_0$ and $X$.
		As $\dim J(L_0,X)\leq n+k+1$ and $\dim T_{p_i}X=n$, we may pick a point $q\in\p^r$ not contained in any of
		\[
			J(L_0,X),T_{p_1}X,\dots,T_{p_s}X.
		\]
		If $m\neq r$, then we furthermore require $q\notin S$.
		We now fix the point $q$.
		
		Let
		\[
			\Xi :=\overline{\{(L,z,x,y)\in \Sigma\times\p^r\times X\times \p^r\mid x\neq y, y\in L, z\in \overline{xy}\}}.
		\]
		Let $\Theta:=p_{12}(\Xi )\subseteq \Sigma\times\p^r$ be the image of the $(1,2)$-th projection. Then
		\[
			\Theta=\overline{\{(L,z)\in \Sigma\times\p^r\mid z\in J(L,X)\}}
		\]
		where $J(L,X)$ denotes the join of $L$ and $X$. Let $\pi_1:\Sigma\times \p^r\to \Sigma$ be the first projection map. 
		Let
		\[
			Y:=(\Sigma\times \{q\})\cap ((\Sigma\times \p^r)\backslash \Theta)
		\]
		which is an open subset of $\Sigma\times \{q\}$.
		Then as $q\notin J(L_0,X)$, we have $(L_0,q)\in Y$. Hence $Y$ is nonempty, and thus it is a dense open subset of $\Sigma\times \{q\}$. Let $\Sigma_1:=\pi_1(Y)$. Then $\Sigma_1$ is a dense open subset of $\Sigma$, and $L_0\in \Sigma_1$. Furthermore, for any $L\in \Sigma_1$, we have $q\notin J(L,X)$ by the definition of $\Sigma_1$. Now let $\Sigma_2:=\Sigma_0\cap \Sigma_1$ which is a dense open subset of $\Sigma$.
		
		Note that $q\notin J(L,X)$ ensures that the projection from $q$ does not introduce a new intersection point of $X$ and $L$ after the projection. In other words, if $p\in\pi_q(L)\cap\pi_q(X)$, then $p=\pi_q(\widetilde{p})$ for some $\widetilde{p}\in L\cap X$. Furthermore, there is no point $x\in X$ such that $\pi_q(x)=p$ unless $x=\widetilde{p}$.
		
		Let $g:X_{\textnormal{sm}}\to\g(n,r)$ be the Gauss map sending $p\in X_{\textnormal{sm}}$ to $T_pX$. Let
		\[
			W:=\{T\in\g(n,r)\mid q\notin T\}.
		\]
		Then $W$ is a nonempty open subset of $\g(n,r)$. Hence $g^{-1}(W)$ is a nonempty open subset of $X$. Let $Z:=X\backslash g^{-1}(W)$ which is a proper closed subset of $X$.
		Define
		\[
			\Sigma_3:=\{L\in \Sigma_2\mid \dim(L\cap Z)\leq k-2\}
		\]
		which is an open subset of $\Sigma_2$ by the upper semicontinuity. Now we will show $L_0\in \Sigma_3$. As $q$ is not contained in any of
		\[
			J(L_0,X),T_{p_1}X,\dots,T_{p_s}X,X,
		\]
		the projection $\pi_q|_X:X\to \p^{r-1}$ maps $p_i$ to a smooth point of $\pi_q(X)$. Hence $\pi_q(H_i)$ has the smooth point $\pi_q(p_i)$ of $\pi_q(X)$. This means the general point in $\pi_q(H_i)$ is a smooth point of $\pi_q(X)$.
		Hence for each $i$ and for the general $p\in H_i$, $q$ is not contained in $T_pX$.
		Hence the locus of $p\in L_0\cap X_{\text{sm}}$ such that $q\in T_pX$ has dimension $\leq k-2$. Recall that we also have $\dim (L_0\cap X_{\text{sing}})\leq k-2$.
		Thus $L_0\in \Sigma_3$. In particular, $\Sigma_3$ is nonempty, hence is a dense open subset of $\Sigma_2$. By the construction, for all $L\in \Sigma_3$, we have
		\[
			\dim (\pi_q(L)\cap \pi_q(X)_{\text{sing}})\leq k-2,
		\]
		and $\pi_q(L)\cap\pi_q(X)$ scheme-theoretically contains a degree $\geq d$ hypersurface of $\pi_q(L)$. Hence $\pi_q(L)$ is an element of $\Sigma^*_{k,d}(\pi_q(X))$ for all $L\in\Sigma_3$.
		
		As $\Sigma_3$ is dense open in $\Sigma$, $ S_3:=\bigcup\limits_{L\in\Sigma_3}L$ has dimension $m$ by \Cref{lem_closure_in_grassmannian}.
		As $\pi_q(L)$ is an element of $\Sigma^*_{k,d}(\pi_q(X))$ for all $L\in\Sigma_3$, we have $\pi_q( S_3)\subseteq S^*_{k,d}(\pi_q(X))$. Therefore the result follows since $\pi_q( S_3)$ has dimension $m$ if $m<r$, and has dimension $r-1$ if $m=r$. Also the choice of $q\in\p^r$ was general in the sense that we only required it to avoid finitely many proper closed subsets of $\p^r$.
	\end{proof}
	
	Finally, we deduce the following bound of $\dim S^*_{k,d}$.

	\begin{thmalphabetmaintext}\label{thm_dim_skd_max}
		Let $X\subseteq\p^r$ be a projective variety of dimension $n$ and let $1\le k\le n+1$ and $d\ge 2$. Then \[\dim S^*_{k,d}(X)\leq \max(n+1+\dfrac{n+1-k}{d-1},n+k).\]
	\end{thmalphabetmaintext}
	\begin{proof}
		We assume $r> \max(n+1+\dfrac{n+1-k}{d-1},n+k)$, or the result follows trivially. We divide into two cases.
		
		For the first case, suppose $n+k+1>n+1+\dfrac{n+1-k}{d-1}$. If $$\dim S^*_{k,d}(X)\geq n+k+1,$$ then we take general projections until $r=n+k+1$. Let $\pi(X)$ be the final image of $X$ under such general projections.
		Then we get $\dim S^*_{k,d}(\pi(X))= n+k+1$ from \Cref{lem_general_projection_preserves_dim_skd}, and we get $\dim S^*_{k,d}(\pi(X))\leq n+k$ from \Cref{thm_dim_skd<=r-1}. This is a contradiction. Hence $\dim S^*_{k,d}(X)\leq n+k$.
		
		For the second case, suppose $n+k+1\leq n+1+\dfrac{n+1-k}{d-1}$. If $$\dim S^*_{k,d}(X)>n+1+\dfrac{n+1-k}{d-1},$$ then we take general projections until $r$ satisfies \[n+1+\dfrac{n+1-k}{d-1}<r\leq n+2+\dfrac{n+1-k}{d-1}.\] Let $\pi(X)$ be the final image of $X$ under such general projections. Then we get $\dim S^*_{k,d}(\pi(X))= r$ from \Cref{lem_general_projection_preserves_dim_skd}, and we get $\dim S^*_{k,d}(\pi(X))\leq r-1$ from \Cref{thm_dim_skd<=r-1}. This is a contradiction. Hence $\dim S^*_{k,d}(X)\leq n+1+\dfrac{n+1-k}{d-1}$.
	\end{proof}

	\begin{example}\label{example:bound_sharp_Vero}
		Let $X=\nu_2(\p^2)\subseteq\p^5$ be the Veronese surface.
		Then we get $\dim S_{2,2}(X)\le 4$ by \Cref{thm_dim_skd_max}.
		As $S_{2,2}(X)=\sigma_2(X)$ is a fourfold, the equality of the bound \Cref{thm_dim_skd_max} holds in this case.
	\end{example}
	
	\section{Applications}\label{sec:applications}
	\subsection{A lower bound on the number of linear syzygies}\label{subsec_lower_bound_betti}
	
	In this subsection, we give a lower bound on the number of independent linear syzygies of $d$-forms satisfying $\K_d$ and give some examples. As before, all results in this subsection for projective varieties satisfying $\K_d$ also apply to the projective varieties satisfying \textbf{N$_{d,2}$} (\Cref{prop_Nd2_implies_Kd}).
	
	\begin{remark}\label{rem_lower_bound_linear_syzygies}
		Let $X\subseteq\p^r$ be a projective variety satisfying $\K_d$ by $d$-forms $F_0,\dots,F_s$. If $\dim S_{k,d}(X)\leq r-1$, then there are at least $r+1-k$ independent linear syzygies of $F_0,\dots,F_s$. Indeed, if there are at most $r-k$ independent linear syzygies of $F_0,\dots,F_s$, then we have $S_{k,d}(X)\cup X=\p^r$ by \Cref{thm_equations_of_skd}, which is a contradiction.
	\end{remark}
	
	\begin{remark}\label{rem_skd_not_full_necessary}
		The condition $S_{k,d}(X)\neq \p^r$ in \Cref{rem_lower_bound_linear_syzygies} is necessary since without this condition, there can be fewer than $r+1-k$ linear syzygies. For example, the rational normal curve $C\subseteq\p^3$ has the Betti number $\beta_{2,1}(C)=2<3+1-1$, and $\g(1,4)\subseteq\p^9$ has the Betti number $\beta_{2,1}(\g(1,4))=5<9+1-1$ where both varieties satisfy $\N_{2,2}$, hence also $\K_2$.
	\end{remark}
	
	Using the dimension bound of $S^*_{k,d}(X)$ in \Cref{thm_dim_skd<=r-1}, we get the following corollary.

	\begin{coralphabetmaintext}\label{cor_lower_bound_beta_2d-1}
		Let $X\subseteq\p^r$ be a projective variety of dimension $n$ and let $1\le k\le n+1$ and $d\ge 2$.
		Suppose that the homogeneous ideal $I_X$ is generated by $d$-forms satisfying $\N_{d,2}$.
		If $r>n+1+\dfrac{n+1-k}{d-1}$ and $\dim X_{\sing}\leq k-2$, then $\beta_{2,d-1}(X)\geq r+1-k$.
	\end{coralphabetmaintext}
	\begin{proof}
		If $r>n+1+\dfrac{n+1-k}{d-1}$, then $\dim S^*_{k,d}(X)\leq r-1$ by \Cref{thm_dim_skd<=r-1}. Since $\dim X_\sing\leq k-2$, we also have $\dim S_{k,d}(X)\leq r-1$ since $S_{k,d}(X)=S^*_{k,d}(X)$ (\Cref{rem_skd=skd*_if_sing_small}). Hence the result follows from \Cref{rem_lower_bound_linear_syzygies}.
	\end{proof}
	
	The following gives examples $X$ that attain the bound in \Cref{cor_lower_bound_beta_2d-1}.
	
	\begin{example}[A variety $X$ for which $\beta_{2,d-1}(X)$ attains the lower bound]\label{exp_sharp_cases_of_lower_bound}
		Let $2\leq d\leq 5$ be an integer and $\seg(\p^{d-1}\times\p^{d})\subseteq\p^{d^2+d-1}$ be the embedded Segre variety and \[X=\sigma_{d-1}(\seg(\p^{d-1}\times\p^{d}))\] be its $(d-1)$-secant variety.
		Then $\dim X=d^2+d-3$ and $\dim X_{\text{sing}}=d^2+d-7$. Also the minimal free resolution of $X$ is given by the Eagon--Northcott complex with the following Betti table.
		\[\begin{array}{c|ccc}
			&0 &1 &2 \\ \hline
			0&1 & & \\
			d-1& &d+1 &d 
		\end{array}\]
		Consider the function $f(n,k,d):=\lfloor n+1+\dfrac{n+1-k}{d-1}\rfloor$ where $n=\dim X=d^2+d-3$.
		Then
		\begin{itemize}
			\item $f(n,k,d)=d^2+d-2$ when $k=d^2$,
			\item $f(n,k,d)\ge d^2+d-1$ when $k<d^2$.
		\end{itemize}
		Hence the lower bound of $\beta_{2,d-1}(X)$ obtained from \Cref{cor_lower_bound_beta_2d-1} is $d$. Here, the condition $d\le 5$ guarantees that $\dim X_\sing\le d^2-2$.
		Since $\beta_{2,d-1}(X)=d$, this example attains the lower bound in \Cref{cor_lower_bound_beta_2d-1}.		
	\end{example}

	\begin{remark}\label{rem_smooth_threefold_in_p5}
		Let $X\subseteq\p^5$ be a smooth threefold whose homogeneous ideal $I_X$ is minimally generated by quartic forms satisfying $\N_{4,2}$.
		Then we get $\beta_{2,3}(X)\geq 4$ from \Cref{cor_lower_bound_beta_2d-1}. Note that this is obtained from the bound $\dim S_{2,4}(X)\leq 4$ (\Cref{thm_dim_skd<=r-1}). Also if $\beta_{2,3}(X)=4$, then $\dim S_{1,4}(X)=5$ by \Cref{rem_lower_bound_linear_syzygies}.
	\end{remark}
	
	\subsection{Double point loci and $S_{k,2}(X)$}
	Let $\pi_p:X\to\p^{r-1}$ be the projection of $X$ from the general point $p\in\p^r\backslash X$.
	The following proposition gives another characterization of $S_{k,d}(X)$ in terms of projections.
	
	\begin{proposition}[Characterization of $S_{k,d}(X)$ in terms of the projection]\label{prop_Skd_characterization}
		Let $X\subseteq\p^r$ be a projective variety and let $k\ge 1$ and $d\ge 2$. Then for a point $p\in\p^r\backslash X$, it holds that $p\in S_{k,d}(X)$ if and only if the $d$-ple point locus \[Z_d(p,X)=\{x\in\p^{r-1}\mid \leng\ \pi_p^{-1}(x)\geq d\}\]
		contains a $(k-1)$-plane.
	\end{proposition}
	\begin{proof}
		We only show the converse as the other direction is clear. Suppose the $d$-ple point locus of the projection $\pi_p:X\to\p^{r-1}$ contains a $(k-1)$-plane $\Lambda$. Then every line in the $k$-plane $L:=\langle p,\Lambda\rangle$ passing through $p$ is a $d$-secant line of $X$. If $\dim (L\cap X)\leq k-2$, then the $d$-secant lines in $L$ through $p$ form a family of dimension at most $k-2$, which is a contradiction. Also $L\cap X\neq L$ since $p\notin X$. Thus $\dim (L\cap X)=k-1$. If the $(k-1)$-dimensional part of $L\cap X$ is a hypersurface of degree $\leq d-1$ in $L$, then not every line in $L$ passing through $p$ is a $d$-secant line of $X$. Therefore $p\in S_{k,d}(X)$.
	\end{proof}

	Then we get a depth formula of projections from \Cref{prop_Skd_characterization}.
	\begin{corollary}\label{cor_depth_formula}
		Let $X\subseteq\p^r$ be a nondegenerate arithmetically Cohen--Macaulay (ACM) projective variety of codimension $\ge 2$ satisfying $\N_{2,2}$. Define $M$ as in \Cref{intro_set-theoretic_eq}, and denote the projection of $X\subseteq\p^r$ from a point $p\in \p^r\backslash X$ as $X_p\subseteq\p^{r-1}$.
		Then we have
		\[
		\depth X_p=r+1-\rank M(p).
		\]
	\end{corollary}
	\begin{proof}
		By \Cref{prop_Xp_using_secant_locus}, we have
		\begin{align*}
			\depth X_p &=\dim\wt{Z}_2(p,X)+2\\
			&=\dim Z_2(p,X)+2
		\end{align*}
		since $X$ is ACM of codimension $\ge 2$.
		If $p\notin \left(\sigma_2X\cup \tan(X)\right)\backslash X$, then $Z_2(p,X)=\varnothing$ and we get $\depth X_p=1$.
		As $p\notin S_{1,2}(X)\cup X$, we get $\rank M(p)=r$ by \Cref{thm_equations_of_skd}, hence the statement follows.
		Assume $p\in \left(\sigma_2X\cup \tan(X)\right)\backslash X$.
		As $X$ satisfies $\N_{2,2}$, the double point locus $Z_2(p,X)$ is empty or a linear space by \Cref{thm_d-ple_linear}.
		Hence by \Cref{prop_Skd_characterization},
		\[
			\dim Z_2(p,X)=\max\Big(\{0\}\cup\{k\in \z\mid p\in S_{k,2}(X)\}\Big)-1.
		\]
		The right hand side is equal to $r-\rank M(p)-1$ by \Cref{thm_equations_of_skd}. Therefore the result follows.
	\end{proof}
	
	\begin{example}\label{exp_3-minors_of_3x4_generic_matrix_for_general_projection}
		Let $\seg(\p^2\times \p^3)\subseteq\p^{11}$ be the embedded Segre variety and $X:=\sigma_2(\seg(\p^2\times \p^3))$ be its 2-secant variety.
		Then $\dim X=9$ and $\dim X_{\text{sing}}=5$. Also, it has the following Betti table.
		\[\begin{array}{c|ccc}
			&0 &1 &2 \\ \hline
			0&1 & & \\
			1& & & \\
			2& &4 &3 
		\end{array}\]
		From this, we immediately get $\dim S_{8,3}(X)=11$ since if $\dim S_{8,3}(X)\leq 10$, then $\beta_{2,2}(X)\geq 4$ by \Cref{rem_lower_bound_linear_syzygies}.
		Also, since $r=11>9+1+\dfrac{9+1-9}{3-1}$ we have $\dim S_{9,3}(X)=\dim S^*_{9,3}(X)\leq 10$ by \Cref{thm_dim_skd<=r-1}.
		Hence for the general point $p\in\p^{11}\backslash X$ and the projection $\pi_p:X\to \p^{10}$, the triple point locus $Z_3(p,X)$ should contain a $7$-plane and it should not contain an $8$-plane. Also it is a linear space by \Cref{thm_d-ple_linear}. Hence $Z_3(p,X)=\p^7$.
		
		Let $M$ be the matrix obtained as in \Cref{thm_equations_of_skd}.
		Since $M$ has the size $3\times 12$, we have
		\[
			I_4(M)=I_5(M)=\cdots=I_{11}(M)=0.
		\]
		Hence
		\[
			S_{1,3}(X)=S_{2,3}(X)=\cdots=S_{8,3}(X)=\p^{11}.
		\]
		Also we get $I_3(M)\neq 0$ and $V(I_3(M))=X$ so that $S_{9,3}(X)\subseteq X$ using Macaulay2. This implies that the triple point locus $Z_3(p,X)$ of any outer projection is a $7$-plane by \Cref{prop_Skd_characterization}.
	\end{example}
	
	\begin{example}\label{exp_g(14)_for_general_projection}
		The 6-dimensional smooth projective variety $X:=\g(1,4)\subseteq\p^9$ has the following Betti table.
		\[\begin{array}{c|cccc}
			&0 &1 &2 &3 \\ \hline
			0&1 & & & \\
			1& &5 &5 & \\
			2& & & &1
		\end{array}\]
		From this, we get $\dim S_{4,2}(X)=9$ since if $\dim S_{4,2}(X)\leq 8$, then $\beta_{2,1}(X)\geq 6$ by \Cref{rem_lower_bound_linear_syzygies}.
		Furthermore, since $r=9>6+1+\dfrac{6+1-6}{2-1}$, we have $\dim S_{6,2}(X)=\dim S^*_{6,2}(X)\leq 8$ by \Cref{thm_dim_skd<=r-1}.
		Hence for the general point $p\in\p^9\backslash X$ and the projection $\pi_p:X\to \p^8$, the double point locus $Z_2(p,X)$ should contain a $3$-plane and it should not contain a $5$-plane. Also it is a linear space by \Cref{thm_d-ple_linear}. Hence $Z_2(p,X)$ is either a $3$-plane or a $4$-plane.
		
		Let $M$ be the matrix obtained as in \Cref{thm_equations_of_skd}.
		Since $M$ has the size $5\times 10$, we have
		\[
			I_6(M)=\cdots=I_9(M)=0,
		\]
		and one can check $I_5(M)=0$ using Macaulay2.
		Hence
		\[
			S_{1,2}(X)=S_{2,2}(X)=\cdots=S_{5,2}(X)=\p^9.
		\]
		We also get $I_4(M)\neq 0$ and $V(I_4(M))=X$ so that $S_{6,2}(X)\subseteq X$.
		Thus the double point locus $Z_2(p,X)$ of any outer projection is a 4-plane by \Cref{prop_Skd_characterization}.
		Furthermore, by the proof of \Cref{thm_equations_of_skd}, the inclusion $S_{6,2}(X)\subseteq X$ implies $\rank M(p)=4$.
		This implies $\depth X_p=6$ by \Cref{cor_depth_formula}.
	\end{example}

	\begin{example}
		Let $X=S(2,3)\subseteq\p^6=\proj\c[x_0,\dots,x_6]$ be the rational normal scroll defined by $2$-minors of the matrix
		\[\begin{bmatrix}
			x_0 & x_1 & x_3 & x_4 & x_5\\
			x_1 & x_2 & x_4 & x_5 & x_6
		\end{bmatrix}_{\textstyle\raisebox{4pt}.}\]
		Then
		\[
			(\sigma_2X)_\sing=V(x_3,x_4,x_5,x_6)\cup X=S_{2,2}(X)\cup X.
		\]
		Here, the $2$-plane $L:=V(x_3,x_4,x_5,x_6)$ is the linear span of the plane conic $S(2)\subseteq S(2,3)$.
		Hence $L\subseteq S_{2,2}(X)$.
		In general, let $X\subseteq\p^r$ be a nondegenerate smooth projective variety satisfying $\N_{2,2}$. 
		By considering the map from the secant bundle $B$ (see e.g. \cite{MR3436394}) to $\sigma_2X$ which is bijective off $S_{2,2}(X)\cup X$, one might expect
		\[
			(\sigma_2X)_\sing\subseteq S_{2,2}(X)\cup X.
		\]
		However, the equality need not hold in this inclusion.
		For example, the Veronese surface $\nu_2(\p^2)\subseteq\p^5$ satisfies $(\sigma_2(\nu_2\p^2))_\sing=\nu_2\p^2$ while $S_{2,2}(\nu_2\p^2)\cup \nu_2\p^2=\sigma_2(\nu_2\p^2)$.
	\end{example}

	\section{Multisecant lines to the varieties of small codimension}	\label{sec_Multisecant Lines to the Varieties of Codimension Two}
	
	For the codimension two case, it is well studied whether the multisecant lines of a smooth projective variety fill up the ambient projective space as described in the introduction.
	
	If a threefold in \Cref{question:conjecture}(1) exists, it should be found among smooth non-ACM threefolds since, for an ACM projective variety $X\subseteq\p^{n+2}$ of codimension 2, it is known that $(I_X)_{d}=0$ if and only if $S_{1,d+1}(X)=\p^{n+2}$ for all $d\leq n$. The proofs can be found in \cite[Proposition 3.7]{MR1857133} and \cite[Proposition 6.3]{MR2601633}.
	
	However, we found a threefold regular in codimension $1$ whose 4-secant lines do not fill up $\p^5$.
	The following remark was used when we were finding the examples. The examples are constructed via binomial maps between projective spaces.
	
	\begin{remark}\label{rem_determine_probabilistically}
		Although it takes time to compute the minors of $M$, we can determine whether $S_{1,d}(X)$ is equal to $\p^r$ by choosing a point that is a certificate of $S_{1,d}(X)\ne \p^r$. Pick a random point $p\in\p^r$. If $\rank M(p)\geq r$, then $S_{1,d}(X)\neq\p^r$ as $p\notin S_{1,d}(X)\cup X$. Otherwise, if $\rank M(p)\leq r-1$, then $S_{1,d}(X)=\p^r$ with a high probability. This is because if $S_{1,d}(X)\neq\p^r$, then $\dim S_{1,d}(X)\cup X\leq r-1$ so that a randomly selected point in $\p^r$ is rarely contained in $S_{1,d}(X)\cup X$. As computing the rank of a matrix with scalar entries is very fast, one can repeat this process for several points to determine almost surely whether $S_{1,d}(X)$ is equal to $\p^r$ or not.
	\end{remark}
	
	\begin{example}[Threefold $X\subseteq\p^5$ regular in codimension $1$ such that $S_{1,4}(X)\neq \p^5$]\label{example_defective_threefold}
		Let $X\subseteq\p^5$ be the image of the map
		\begin{align*}
			\phi:\proj\c[y_0,\dots,y_3] &\dashrightarrow \proj\c[x_0,\dots,x_5]\\
			[y_0,\dots,y_3] &\mapsto [s_0,\dots,s_5]
		\end{align*}
		where
		\begin{align*}
			s_0&=y_0y_1y_3\\
			s_1&=y_0^2y_2\\
			s_2&=y_0y_2^2\\
			s_3&=y_0^2y_1\\
			s_4&=y_0y_1^2+y_2^2y_3\\
			s_5&=y_0^2y_3+y_2y_3^2.
		\end{align*}
		Then $X$ is a nondegenerate singular threefold of degree $8$ such that $(I_X)_3=0$ with the following Betti table.
		\[\begin{array}{c|cccc}
			&0 &1 &2 &3 \\ \hline
			0&1 & & & \\
			1& & & & \\
			2& & & & \\
			3& &7 &8 &2
		\end{array}\]
		Also, the point $[1,1,1,1,1,1]$ is a certificate of $S_{1,4}(X)\neq\p^5$.
		The set-theoretic equations of the locus $S_{1,4}(X)\cup X$ can be checked in the auxiliary Macaulay2 code.
		The threefold $X$ is regular in codimension $1$ as $X_\sing$ is a union of three lines, but it is not normal.
		We could also check that $\dim(S_{1,4}(X))=4$ and
		\[
			S_{1,4}(X)\cup X=S_{2,4}(X)\cup X.
		\]
		This and \Cref{thm_dim_skd<=r-1} partially explain the reason why $\dim S_{1,4}(X)$ is small in this case.
		\Cref{thm_dim_skd<=r-1} gives
		\[
			\dim S_{2,4}(X')\le 4
		\]
		when a threefold $X'\subseteq\p^5$ has at most isolated singularities.
		However, as $\dim X_{\sing}=1$, this is not a complete explanation.
		We could not compute the radical of these set-theoretic equations.
		Note that $X$ is not ACM since the projective dimension is greater than the codimension.
	\end{example}
	
	When the codimension is greater than two, not many things are known. For example, for $n\geq 2$, no smooth variety $X$ of dimension $2n$ in $\p^{3n+1}$ is known such that $(I_X)_2=0$ and $S_{1,3}(X)\neq\p^{3n+1}$, cf. \cite[p. 575]{MR3024093}. For this case, we found a normal example as follows.
	
	\begin{example}[Normal fourfold $X\subseteq\p^7$ such that $S_{1,3}(X)\neq \p^7$]\label{example_defective_fourfold}
		Let $X\subseteq\p^7$ be the image of the map
		\begin{align*}
			\phi:\proj\c[y_0,\dots,y_4] &\dashrightarrow \proj\c[x_0,\dots,x_7]\\
			[y_0,\dots,y_4] &\mapsto [s_0,\dots,s_7]
		\end{align*}
		where
		\begin{align*}
			s_0&=y_0y_2^2,\\
			s_1&=y_2y_3y_4,\\
			s_2&=y_1^3,\\
			s_3&=y_0y_2y_4,\\
			s_4&=y_1y_2y_3,\\
			s_5&=y_1^2y_2,\\
			s_6&=y_2^2y_4,\\
			s_7&=y_0y_1y_2+y_3^2y_4.
		\end{align*}
		Then $X$ is a nondegenerate normal fourfold of degree $11$ such that $(I_X)_2=0$ with the following Betti table.
		\[\begin{array}{c|cccc}
			&0 &1 &2 &3 \\ \hline
			0&1 & & & \\
			1& & & & \\
			2& &9 &12 &3 \\
			3& & & &1
		\end{array}\]
		Also, the point $[1,1,1,1,1,1,1,1]$ is a certificate of $S_{1,3}(X)\neq\p^7$.
		The set-theoretic equations of the locus $S_{1,3}(X)\cup X$ can be checked in the auxiliary Macaulay2 code which also checks $\dim(S_{1,3}(X))=6$.
		We could not compute the radical of these set-theoretic equations.
		In this case, the locus $X_\sing$ is a union of a line, three $2$-planes, and a hyperquadric lying on a $3$-plane so that $\dim X_{\sing}=2$.
		Note that $X$ is ACM since the projective dimension is equal to $\codim X$.
	\end{example}

	In summary, we conclude the following.
	
	\setcounter{thmalphabetmaintext}{1}
	\begin{coralphabetmaintext}
		There exist the following nondegenerate projective varieties:
		\begin{enumerate}
			\item a threefold $X\subseteq\p^5$ regular in codimension $1$ with $S_{4}(X)\neq \p^5$ and $(I_X)_3=0$,
			\item a normal fourfold $X\subseteq\p^7$ with $S_{3}(X)\neq\p^7$ and $(I_X)_2=0$.
		\end{enumerate}
	\end{coralphabetmaintext}

\bibliographystyle{amsalpha}
\bibliography{ref.bib}
\end{document}